\documentclass[10pt,a4paper,reqno]{amsart}
\usepackage{amsthm,amssymb,amstext,graphicx,comment}
\usepackage{bbm}
\usepackage{enumerate}
\usepackage{amscd}
\usepackage{amssymb}
\usepackage{amsthm}
\usepackage{mathtools}
\usepackage{hyperref}
\usepackage{enumitem}
\usepackage{empheq}
\usepackage[utf8]{inputenc}
\usepackage[T1]{fontenc}
\usepackage{lmodern}
\usepackage{bibentry}
\usepackage{soul}
\usepackage[capitalise]{cleveref}

\usepackage{hyperref}
\hypersetup{
  colorlinks=true,
  pdftitle={},
  pdfauthor={© Sven Karbach, Martin Friesen},
  pdfsubject={Long time behavior affine processes},
  pdfkeywords={}
}

\DeclareMathOperator\dom{dom}
\DeclareMathOperator\Tr{Tr}

\DeclareMathOperator\lin{lin}

\newcommand*\D{\mathop{}\!\textnormal{d}}

\newcommand*\E{\mathop{}\!\textnormal{e}}
\newcommand*\I{\mathop{}\!\textnormal{i}}

\numberwithin{equation}{section}

\newtheorem{theorem}{Theorem}[section]
\newtheorem*{theorem*}{Theorem}
\newtheorem{lemma}[theorem]{Lemma}
\newtheorem{proposition}[theorem]{Proposition}
\newtheorem*{proposition*}{Proposition}
\newtheorem{corollary}[theorem]{Corollary}
\theoremstyle{plain}

\theoremstyle{definition}
\newtheorem{definition}[theorem]{Definition}
\newtheorem{remark}[theorem]{Remark}

\newtheoremstyle{example}
  {.3\baselineskip}
  {.3\baselineskip}
  {\normalsize}  
  {0pt}       
  {\bfseries} 
  {.}         
  {5pt plus 1pt minus 1pt} 
  {}          
\theoremstyle{example}
\newtheorem{example}[theorem]{Example}

\newtheorem*{assumption*}{\assumptionnumber}
\providecommand{\assumptionnumber}{}
\makeatletter
\newenvironment{assumption}[1]
 {%
  \renewcommand{\assumptionnumber}{Assumption $\mathfrak{#1}$}%
  \begin{assumption*}%
  \protected@edef\@currentlabel{$\mathfrak{#1}$}%
 }
 {%
  \end{assumption*}
 }
\makeatother

\newlist{theoremenum}{enumerate}{1}
\setlist[theoremenum]{label=\roman*), ref=\textup{\thetheorem~\roman*)}}
\crefalias{theoremenumi}{theorem}

\newlist{defenum}{enumerate}{1}
\setlist[defenum]{label=\roman*), ref=\textup{\thedefinition~\roman*)}}
\crefalias{defenumi}{definition}

\def\e{\operatorname{e}} 
\def\eps{\varepsilon}
\renewcommand{\epsilon}{\eps}

\renewcommand{\MR}{\mathbb{R}}
\newcommand{\MC}{\mathbb{C}}

\newcommand{\MN}{\mathbb{N}}
\newcommand{\MP}{\mathbb{P}}
\newcommand{\MQ}{\mathbb{Q}}

\newcommand{\MS}{\mathbb{S}}
\newcommand{\MF}{\mathbb{F}}

\newcommand{\R}{\MR}
\newcommand{\N}{\MN}

\newcommand{\cF}{\mathcal{F}}
\newcommand{\cA}{\mathcal{A}}
\newcommand{\cB}{\mathcal{B}}
\newcommand{\cC}{\mathcal{C}}
\newcommand{\cD}{\mathcal{D}}

\newcommand{\cH}{\mathcal{H}}

\newcommand{\cK}{\mathcal{K}}
\newcommand{\cL}{\mathcal{L}}

\newcommand{\cM}{\mathcal{M}}

\newcommand{\cQ}{\mathcal{Q}}

\newcommand{\df}{\coloneqq}

\newcommand{\one}{\mathbf{1}}
\newcommand{\interior}[1]{({\kern0pt#1})^{\textnormal{o}}}
\newcommand{\set}[1]{\left\{ #1\right\}}
\newcommand{\norm}[1]{\|#1\|}

\newcommand{\EXspec}[2]{\mathbb{E}_{#1}\left[#2\right]}
\newcommand{\CEX}[2]{\mathbb{E}\left[#1\lvert\;#2\right]}

\newcounter{Task}
\newcommand{\sgc}{\begin{color}{blue}}
\newcommand{\cgs}{\end{color}}
\newcommand{\cHplus}{\cH^{+}}

\newcommand{\cHpluso}{\cHplus\setminus \{0\}}
\newcommand{\MRplus}{\MR^{+}}
\newcommand{\dm}{m(\D\xi)}
\newcommand{\dmu}{\mu(\D\xi)}

\newcommand{\Bsq}{\cB_{\rho}(\cHplus_{\mathrm{w}})}

\begin{document}
\title[]{Stationary covariance regime for affine stochastic covariance models in Hilbert spaces}  
\author{Martin Friesen}
\address[Martin Friesen]{School of Mathematical Sciences, Dublin City University, Dublin 9, Ireland}
\email{martin.friesen@dcu.ie}

\author{Sven Karbach$^{\ast}$}
\address[Sven Karbach]{Korteweg-de Vries Institute for Mathematics, University of Amsterdam, Postbus 94248, NL–1090 GE Amsterdam, The Netherlands}
\email{sven@karbach.org}
\thanks{$^{\ast}$The author is funded by The Dutch Research Council
  (NWO) (Grant No: C.2327.0099)}

\begin{abstract}
 We study the long-time behavior of affine processes on positive self-adjoiont Hilbert-Schmidt operators which are of pure-jump type, conservative and have finite second moment. For subcritical processes we prove the existence of a unique limit distribution and construct the corresponding stationary affine process. Moreover, we obtain an explicit convergence rate of the underlying transition kernels
 to the limit distribution in the Wasserstein distance of order $p\in [1,2]$ and provide explicit formulas for the first two moments of the limit distribution. 
 We apply our results to the study of infinite-dimensional affine stochastic covariance models in the stationary covariance regime, where the stationary affine process models the instantaneous covariance process. In this context we investigate the behavior of the implied forward volatility smile for large forward dates in a geometric affine forward curve model used for the modeling of forward curve dynamics in fixed income or commodity markets formulated in the Heath-Jarrow-Morton-Musiela framework.
\end{abstract}

\maketitle

\section{Introduction}\label{sec:introduction}
     
Let $\left((H,(\cdot,\cdot)\right)$ be a real separable Hilbert space
and denote by $\left(\cH,\langle\cdot,\cdot\rangle\right)$ the Hilbert space
of all self-adjoint Hilbert-Schmidt operators on $H$ equipped with the trace
inner-product $\langle x,y\rangle\df\Tr(yx)$ for $x,y\in\cH$. Let $\cHplus \subset \cH$ stand for the cone of all positive self-adjoint Hilbert-Schmidt operators on $H$.
Given a time-homogeneous Markov process $(X_{t})_{t\geq 0}$ with values in $\cHplus$, we denote its transition kernels by $(p_{t}(x,\cdot))_{t\geq 0}$. Following the general terminology of affine processes, we call $(X_{t})_{t\geq 0}$
an \textit{affine} process on $\cHplus$, whenever the Laplace transform of
$X_{t}$, for every $t\geq 0$, is of an exponential affine form in the initial value $x\in\cHplus$, i.e. 
\begin{align}\label{eq:affine-transform-intro}
\int_{\cHplus}\E^{-\langle \xi,u\rangle}\,p_{t}(x,\D\xi)=\E^{-\phi(t,u)-\langle
  x,\psi(t,u)\rangle},\quad t\geq 0,\, u\in\cHplus\,, 
\end{align}
for some functions $\phi\colon \MRplus\times\cHplus\to\MRplus$ and $\psi\colon
\MRplus\times \cHplus \to \cHplus$. 
Affine processes are widely used for applications in finance due to their analytical tractability, i.e. their Fourier-Laplace transforms 
given by \eqref{eq:affine-transform-intro} are quasi-explicit up to the functions $\phi$ and $\psi$. Typically the functions $\phi$ and $\psi$ are the solutions of certain differential equations. 

In finite dimensions affine processes and their applications were studied by many authors on numerous state spaces including
the canonical state space $\R_+^d \times \R^n$, see, e.g. \cite{DFS03, keller2008affine, DL06, FM09, KM15, KMK10}, 
and the space of positive and symmetric $d \times d$-matrices $\MS_{d}^{+}$, see \cite{CFMT11, cuchiero2011affine}.
More recently it became increasingly popular to study their infinite-dimensional extensions, see, e.g. \cite{CT20, STY20, Gra16, cox2020affine, cox2021infinitedimensional}. In the present work we contribute to this new direction of research towards affine processes on infinite-dimensional state spaces by studying their long-time behavior. 

The affine processes on the state-space $\cHplus$, as considered in this work, have been first introduced and studied in \cite{cox2020affine}. Their construction is summarized in Theorem~\ref{thm:CKK20} below, see also~\cite[Theorem 2.8]{cox2020affine}. We want to emphasize here that affine processes on $\cHplus$ can be considered as the natural infinite-dimensional analog of affine processes on positive and symmetric $d\times d$-matrices $\MS_{d}^{+}$ studied in~\cite{CFMT11}. Indeed, for $H=\MR^{d}$ the self-adjoint Hilbert-Schmidt operators on $H$ are precisely the symmetric matrices equipped with the Frobenius norm and $\cHplus=\MS_{d}^{+}$ holds in this case. Due to their flexibility and tractability, affine processes on $\MS_d^+$ are widely used to model the instantaneous covariance process in multi-variate stochastic volatility models, see~\cite{CFMT11, LT08, BNS07, GS10}. Analogously, affine processes on the state-space $\cHplus$ are well-suited to model the instantaneous covariance process in infinite-dimensional stochastic covariance models, see~\cite{cox2021infinitedimensional}. The long-time behavior of affine processes, used to model the instantaneous covariance process, plays an important role in the calibration of stochastic covariance models, see~\cite{AKR16, LMP22}.

\subsection{Contribution and related literature}

In this article we deal with the long-time behavior of affine processes on $\cHplus$. In particular we are interested in the existence and uniqueness of a limit distribution $\pi$ as the time $t$ tends to infinity. Heuristically, if the transition kernels $(p_{t}(x,\cdot))_{t\geq 0}$ converge weakly to some probability measure $\pi$ on $\cHplus$, then $\pi$ is an \textit{invariant measure} for
$(p_{t}(x,\cdot))_{t\geq 0}$, i.e.
\begin{align*}
\int_{\cHplus}p_{t}(x,\D\xi)\,\pi(\D x)=\pi(\D\xi), \quad \text{for all
}t\geq 0 \text{ and } x\in\cHplus.  
\end{align*}
Once the existence of an invariant measure is established, we study the construction of the corresponding stationary process. The latter one allows us to introduce stochastic covariance models in the \emph{stationary covariance regime}, see Section~\ref{sec:affine-sv-models-1} of this work. Within this context we model the implied forward volatility smile in a geometric affine covariance model for forward curve dynamics formulated in the Heath-Jarrow-Morton-Musiela (HJMM) modeling framework, see Section~\ref{sec:model-forw-curve} below. In addition to the above, we are also interested in explicit convergence rate towards the invariant measure. 

Below we provide a short version of our main result on the long-time behavior of affine processes on $\cHplus$.
We call an affine process on $\mathcal{H}^+$ \emph{subcritical}, if its state-dependent drift is negative (i.e. all eigenvalues have strictly negative real-parts), see Assumption~\ref{assump:subcritical-pure-jump} below. 
\begin{theorem*}
  Let $(X_{t})_{t\geq 0}$ be a subcritical affine process on $\cHplus$ with
  transition kernels $(p_{t}(x,\cdot))_{t\geq 0}$, the existence of which is
  guaranteed by Theorem~\ref{thm:CKK20} below. Then the following holds true: 
  \begin{enumerate}
  \item[i)]\label{item:existence-intro} There exists a unique invariant
    measure $\pi$ of $(p_{t}(x,\cdot))_{t\geq 0}$ for all $x\in\cHplus$.
  \item[ii)]\label{item:convergence-rates-intro} For every $x\in\cHplus$ and
    $p\in [1,2]$ the sequence $(p_{t}(x,\cdot))_{t\geq 0}$ converges
    exponentially fast to $\pi$ as $t\to\infty$ in the
    Wasserstein distance of order $p$.
  \item[iii)]\label{item:stationarity-intro} There exists a Markov process
    $(X_{t}^{\pi})_{t\geq 0}$, with transition kernels
    $(p_{t}(x,\cdot))_{t\geq 0}$, such that the distribution of
    $X_{t}^{\pi}$ is equal to $\pi$ for all $t\geq 0$.   
  \end{enumerate}
\end{theorem*}

The long-time behavior of affine processes
on the finite dimensional state spaces $(\MR_{+})^{d}\times\MR^{n}$ and $\MS_{d}^{+}$ for $d,n\in\MN$,
is now mostly well-understood. More precisely, based on the representation by strong solutions of stochastic differential equations, ergodicity was studied in \cite{FJR20} for different Wasserstein distances. By using regularity of transition densities with respect to the Lebesgue measure combined with the Meyn-and-Tweedie stability theory the ergodicity in total variation distances has been studied in \cite{BDLP14, FJ20, JKR19, JKR17, MSV20, FJKR22}. Finally, coupling techniques for affine processes are studied in \cite{Wan12, LM15}. 
Unfortunately, these methods implicitly use the dimension of the state-space and hence do not allow for an immediate extension to infinite dimensional settings. Indeed, for general affine processes on $\mathcal{H}^+$ there does not exist, so far, a pathwise construction. The absence of an infinite-dimensional Lebesgue measure prevents us to effectively use the Meyn-and-Tweedie stability theory (in terms of estimates on the density). Although there exist some extensions of the coupling techniques to infinite dimensional state-spaces (see \cite{Li21} for measure-valued branching processes), these methods seem to be closely related to the measure-valued structure of the process and hence not suitable for our Hilbert space framework.

The most promising method to study the long-time behavior for affine processes in infinite-dimensional settings is therefore based on the convergence of Fourier-Laplace transforms. The latter one requires, in view of the affine-transform formula \eqref{eq:affine-transform-intro}, to study the long-time behavior of the solutions to the generalized Riccati equations $\phi$ and $\psi$. For finite-dimensional state spaces, these ideas have been developed in \cite{GKK10, Kel11, KRM12, JKR20, PS09, FJKR20}. In these works, the existence of an invariant distribution (as well as weak convergence of transition probabilities) is obtained from L\'evys continuity theorem. More recently, such techniques also have been applied in \cite{Fri22} to Dawson-Watanabe superprocesses with immigration which form a specific class of affine processes on the state space of (possibly tempered) measures (so-called measure-valued Markov processes). 

Unfortunately, infinite-dimensional analogues of L\'evys continuity theorem on Hilbert spaces require an additional tightness condition on the transition probabilities (to obtain the existence of a limit distribution and hence invariant measure). For Ornstein-Uhlenbeck processes on Hilbert spaces, this problem can be avoided by taking advantage of their infinite divisibility,
see~\cite{CM87}. Note that such processes form a subclass of affine processes, see also Example~\ref{ex:OU} below. Apart from this, the long-time behavior of affine processes in infinite-dimensions has not been investigated in a systematic way. Our work therefore provides a first general treatment of long-time behavior for affine processes on infinite-dimensional Hilbert spaces as a state space. 

Our methodology for the proof of our main result builds on the ideas taken from ~\cite{FJKR20}, where the long-time behavior of affine processes on $\MS_{d}^{+}$ was studied. Namely, we show that for subcritical affine processes the limits $\lim_{t\to\infty}\phi(t,u)$ and $\lim_{t\to\infty}\psi(t,u)$ exist for every $u\in\cHplus$ and hence the Fourier-Laplace transform of the process (see \eqref{eq:affine-transform-intro}) converges when $t \to \infty$.
To overcome the difficulty related to the absence of a full analogue of Le\'vy continuity theorem, we utilize the 
generalized Feller semigroup approach for the process (see Appendix~\ref{sec:gener-fell-semigr} for a definition). More precisely, we provide uniform bounds on the operator norm of the transition semigroup which allows us
to prove that $\lim_{t \to \infty}P_t f =: \ell(f)$ has a limit for a sufficiently large class of functions $f$. By showing that the limit $\ell$ is a continuous linear functional, we can apply a variant of Riesz representation theorem for generalized Feller semigroups to show that $\ell$ has representation $\ell(f) = \int_{\mathcal{H}^+}f(y)\pi(\D y)$. The measure $\pi$ is the desired unique invariant probability measure. As a byproduct we also obtain weak convergence of transition probabilities in the weak topology on $\mathcal{H}^+$. In the second step we strengthen this convergence by proving estimates on the Wasserstein distance of order $p\in [1,2]$ of the transition probabilities to the invariant measure.
In contrast to the finite-dimensional results in \cite{FJR20, FJKR20}, our new bounds are dimension-free and explicit. As a consequence, we conclude that the transition probabilities converge weakly to the invariant measure in the norm topology on $\mathcal{H}^+$. Finally, we show that the invariant measure has finite second moments and compute them explicitly. 

\subsection{Applications and Examples}

Our main motivation for studying the long-time behavior of affine processes on positive Hilbert-Schmidt operators comes from stochastic covariance modeling in Hilbert spaces, see \cite{BRS18, benth2021barndorff, BS18, cox2021infinitedimensional}. Generally speaking, a stochastic covariance model in an Hilbert space
consist of two processes: one is a Hilbert space valued process that models some stochastic dynamics, e.g. forward curve dynamics in fixed income or commodity
markets formulated in the HJMM framework. The second process models the instantaneous covariance process of the former and takes values in some space of positive and self-adjoint operators, e.g. the cone of positive and self-adjoint Hilbert-Schmidt operators. An affine stochastic covariance model is a stochastic covariance model such that the mixed Fourier-Laplace transform of the joint process satisfies an affine transform formula. 

The affine stochastic covariance model proposed in~\cite{cox2021infinitedimensional} consists of a "price process" determined by a linear SDE in a Hilbert space driven by a Hilbert-valued Brownian motion which is modulated by the square-root of an affine process on positive Hilbert-Schmidt
operators as introduced in~\cite{cox2020affine}. Inspired by the univariate case in \cite{Kel11} we introduce the corresponding affine stochastic covariance model in the \emph{stationary covariance regime}. Namely, we replace the affine instantaneous
covariance process by the stationary affine process with the same transition
probabilities, the existence of which is guaranteed by Corollary~\ref{cor:stationary}. 
In Proposition~\ref{prop:affine-formula-stationary} we show that the affine stochastic covariance model in the stationary covariance regime satisfies an affine transform formula, which makes it a very tractable model for, e.g. pricing options written on forwards, see Section~\ref{sec:model-forw-curve} below. As an example we derive the characteristic function of
the operator valued Barndorff--Nielsen-Shepard (BNS) model established
in~\cite{BRS18} in the stationary covariance regime. This complements the
literature on operator valued BNS type models by their long-time behavior. This was already studied for the matrix valued case in~\cite{BNS07, PS09}.\par{} 

Lastly, we present a geometric affine stochastic covariance model for commodity forward
curve dynamics and consider a class of forward-start options written on forwards. We then study the implied forward volatility smile in the geometric affine stochastic covariance model and show in Proposition~\ref{prop:implied-vol-forward-start}, that it converges to the implied spot volatility of a European call option written on the forward, but this time modeled in the stationary covariance regime. This extends a result in \cite[Proposition 5.2]{Kel11} for forward-start options on (univariate) affine stochastic volatility (SV) models to an infinite-dimensional setting. 
\subsection{Layout of the article}
In Section~\ref{sec:prel-affine-proc} we introduce the class of
affine processes on $\cHplus$ established in~\cite{cox2020affine} and
recall some preliminary results. Subsequently, in
Section~\ref{sec:main-results} we present and discuss our main results in
full detail. Afterwards, in Section~\ref{sec:applications}, we discuss
applications of our results in the context of affine stochastic covariance
models in Hilbert spaces. Finally, the proofs are contained in Section~\ref{sec:proof-main-results},
which is subdivided into several subsections:
first we consider the long-time behavior of the solutions of the generalized
Riccati equations in Section~\ref{sec:gener-ricc-equat}, then prove the
existence of a unique invariant measure in
Section~\ref{sec:invar-distr-affine}, derive the convergence rates in
Section~\ref{sec:conv-rates-wass}, show existence of stationary affine
processes in Section~\ref{sec:proof-coroll-refc} and lastly prove the moment
formulas of the invariant measure in Section~\ref{sec:proof-prop-refpr}.  For the readers convenience we
added some background information on generalized Feller semigroups in
Appendix~\ref{sec:gener-fell-semigr}, where in particular we give a version of
Kolmogorov's extension theorem tailored to our needs. In
Appendix~\ref{sec:some-prop-wass-1} we give a convolution property for the
Wasserstein distance of order $2$. 

\section{Preliminaries: Affine processes on \texorpdfstring{$\cHplus$}{the
    cone of positive self-adjoint Hilbert-Schmidt
    operators}}\label{sec:prel-affine-proc}
    
We set $\MN_{0}=\set{0,1,2,\ldots}$ and $\MN=\set{1,2,\ldots}$. For a complex number $z\in\MC$ we denote its real part by $\Re(z)$ and its
imaginary part by $\Im(z)$. For a vector space $X$ and a subset $U\subseteq X$ we denote the linear
span of $U$ in $X$ by $\lin(U)$. For $(X,\tau)$ a topological vector space and
a subset $S\subseteq X$ we denote the \emph{Borel-$\sigma$-algebra} generated by the
relative topology on $S$ by $\cB(S)$. We write $C_{b}(S)$ 
for the space of real-valued bounded functions on $S$ that are continuous with respect to
the relative topology. $C_{b}(S)$ is a Banach space when endowed with the supremum
norm $\| \cdot \|_{C(S)}$. For a Banach space $X$ with norm $\norm{\cdot}_{X}$ we denote by $\cL(X)$ the
space of all \emph{bounded linear operators} on $X$, which becomes a Banach space when equipped with the operator norm $\| \cdot \|_{\cL(X)}$.\par{} 

Throughout this article we let $(H, \langle \cdot,\cdot\rangle_H)$ be a
separable real Hilbert space. We denote by $\cH$ the set of all self-adjoint
\emph{Hilbert-Schmidt operators} from $H$ to $H$. This is a Hilbert space when
endowed with the \textit{trace inner} product $\langle A, B \rangle =
\sum_{n=1}^{\infty} \langle A f_n, B f_n \rangle_H$ for $A,B\in\cH$ and where $(f_n)_{n\in \MN}$ is an orthonormal basis for $H$. Note that $\langle \cdot, \cdot \rangle$ is
independent of the choice of the orthonormal basis (see, e.g.,~\cite[Section
VI.6]{Wer00}). We denote by $\| \cdot \|$ the norm on $\cH$ induced by
$\langle \cdot, \cdot \rangle$.  In addition, we define by $\cH^+$ the set of all positive operators in $\cH$ i.e. $\cH^{+} \df \{ A \in \cH \colon \langle
Ah, h\rangle_H \geq 0 \text{ for all } h\in H \}$. Note that $\cHplus$ is
a closed subset of $\cH$. Moreover, it is a convex cone in $\cH$
i.e. $\cHplus+\cHplus\subseteq \cHplus$, $\lambda\cHplus \subseteq \cHplus$
for all $\lambda\geq 0$ and $\cHplus\cap(-\cHplus) = \{ 0\}$.
The cone $\cHplus$ induces a partial ordering $\leq_{\cHplus}$ on $\cH$, which
is defined by $x\leq_{\cHplus} y$ whenever $y-x\in \cHplus$. The cone
$\cHplus$ is also \emph{generating} for $\cH$ i.e. $\cH=\cHplus -\cHplus$.
For a Hilbert space $(V,\langle \cdot,\cdot\rangle_{V})$, in this article
either $(H, \langle \cdot,\cdot\rangle_H)$ or $(\cH,\langle
\cdot,\cdot\rangle)$, we denote the \emph{adjoint} of $A\in\cL(V)$ by
$A^{*}$. For two elements $x$ and $y$ in $V$ we define the operator $x\otimes
y\in \cL(V)$ by $x\otimes y(h)=\langle x,h \rangle_{V} y$ for every $h\in V$
and write $x^{\otimes 2}\df x\otimes x$.  

\subsection{Admissible parameters and moment conditions}\label{sec:admiss-param-moment}
Let $\chi\colon \cH\to \cH$ be given by $\chi(\xi)=\xi\one_{\norm{\xi}\leq
  1}(\xi)$. The following definition of an \textit{admissible parameter set}
stems from \cite[Definition 2.3]{cox2020affine}.     
\begin{definition}\label{def:admissibility}
  An \emph{admissible parameter set} $(b,B,m,\mu)$ consists of
  \begin{defenum}
  \item\label{item:m-2moment} a measure
    $m\colon\cB(\cHpluso)\to [0,\infty]$ such that
    \begin{enumerate}
    \item[(a)] $\int_{\cHpluso} \| \xi \|^2 \,\dm < \infty$ and
    \item[(b)] $\int_{\cHpluso}|\langle \chi(\xi),h\rangle|\,\dm<\infty$ for all $h\in\cH$  
      and there exists an element $I_{m}\in \cH$ such that $\langle
      I_{m},h\rangle=\int_{\cHpluso}\langle \chi(\xi),h\rangle\, m(\D\xi)$ for every $h\in\cH$; 
    \end{enumerate}   
  \item\label{item:drift} a vector $b\in\cH$ such that
    \begin{align*}
      \langle b, v\rangle - \int_{\cHpluso} \langle
      \chi(\xi), v\rangle \,m(\D\xi) \geq 0\, \quad\text{for all}\;v\in\cHplus;
    \end{align*}    
  \item\label{item:affine-kernel} a $\cH^{+}$-valued measure\footnote{See
      \cite{BDS55} and \cite{Lew70} for the notion of a vector valued measure
      and integration theory with respect to these.}
    $\mu \colon \mathcal{B}(\cHpluso) \rightarrow \cH^+$ such that the kernel
    $M(x,\D\xi)$, for every $x\in\cHplus$ defined on $\mathcal{B}(\cHpluso)$ by
    \begin{align}\label{eq:affine-kernel-M}
      M(x,\D\xi)\df \frac{\langle x, \mu(\D\xi)\rangle }{\norm{\xi}^{2}},
    \end{align}
    satisfies, for all $u,x\in \cH^{+}$ such that $\langle u,x \rangle = 0$,
    \begin{align}\label{eq:affine-kernel-quasi-mono}
      \int_{\cH^+\setminus \{0\}} \langle \chi(\xi), u\rangle M(x,\D\xi)< \infty; 
    \end{align}
  \item\label{item:linear-operator} an operator $B\in \mathcal{L}(\mathcal{H})$ 
    with adjoint $B^{*}$ satisfying
    \begin{align*}
      \left\langle B^{*}(u) , x \right\rangle 
      - 
      \int_{\cHpluso}
      \langle \chi(\xi),u\rangle 
      \frac{\langle \dmu, x \rangle}{\| \xi\|^2 }
      \geq 0,
    \end{align*}
    for all $x,u \in \cHplus$ with $\langle u,x\rangle=0$. 
  \end{defenum}
\end{definition}
 \begin{remark}\label{rem:admissible-effective}
  \begin{enumerate}
  \item[1)] Part (a) of \cref{item:m-2moment} yields
    $\int_{\cHplus\cap\{\norm{\xi}>1\}}\norm{\xi}\,m(\D\xi)\leq
      \int_{\cHplus\cap\{\norm{\xi}>1\}}\norm{\xi}^{2}\,m(\D\xi)<\infty$. Hence the integral $\int_{\cHplus\cap\{\norm{\xi}>1\}}\xi\,m(\D\xi)$ is
    well-defined in the Bochner sense.
  \item[2)] Similarly, the map
    $u\mapsto\int_{\cHplus\cap\set{\norm{\xi}>1}}\langle \xi,u\rangle
    \frac{\mu(\D\xi)}{\norm{\xi}^{2}}$ is a bounded linear operator on
    $\cH$. Indeed, for $v\in\cH$ such that $v=v^{+}-v^{-}$ for $v^{+},v^{-}\in\cHplus$
    we write 
    $|\langle\mu(\D\xi),v\rangle|\df\langle\mu(\D\xi),v^{+}\rangle+\langle\mu(\D\xi),v^{-}\rangle$
    and see that $|\langle\mu(\D\xi),v\rangle|$ is a positive measure for all
    $v\in\cH$. We thus have:
    \begin{align*}
     \Big\langle\int_{\cHplus\cap\set{\norm{\xi}>1}}\langle \xi,u\rangle\frac{
      \mu(\D\xi)}{\norm{\xi}^{2}},v\Big\rangle&\leq\norm{u}\Big(
     \int_{\cHplus\cap\set{\norm{\xi}>1}}\norm{\xi}^{-1}|\langle{\mu(\D\xi)},v\rangle|\Big)\\
                                        &\leq \norm{u} |\big\langle\mu(\cHplus\cap\set{\norm{\xi}>1}),v\big\rangle|,  
    \end{align*}
    taking the supremum over all $v\in\cH$ with $\norm{v}=1$ on both sides
    proves the boundedness of the map. Note
    that $\norm{\mu(\cHplus\cap\set{\norm{\xi}>1})}<\infty$, since by
    \cref{item:affine-kernel} for every $A\in\cB(\cHpluso)$ we have
    $\mu(A)\in\cHplus$ and hence $\norm{\mu(A)}<\infty$. 
  \end{enumerate}
\end{remark}

Let $u\in\cH$. We define the constant and linear \textit{effective drift} terms
$\hat{b}$ and $\hat{B}$ as follows:   
\begin{align}\label{eq:b-B-check}
  \hat{b}\df b+\int_{\cHplus\cap\{\norm{\xi}>1\}}\xi\,m(\D\xi) \quad\text{and}\quad \hat{B}(u)\df B^{*}(u)+\int_{\cHplus\cap\set{\norm{\xi}>1}}\langle
  \xi,u\rangle \frac{\mu(\D\xi)}{\norm{\xi}^{2}}.  
\end{align} 
By the preceding Remark~\ref{rem:admissible-effective} we see that
$\hat{b}\in\cH$ and $\hat{B}\in\cL(\cH)$ are well-defined.\par{}
Now, we recall the main result in \cite[Theorem 2.8]{cox2020affine} which
ensures the existence of a broad class of affine processes on $\cHplus$ associated with the admissible parameter set $(b,B,m,\mu)$ satisfying the
conditions in Definition~\ref{def:admissibility}.  
\begin{theorem}\label{thm:CKK20}
Let $(b, B, m, \mu)$ be an admissible parameter set according to
Definition~\ref{def:admissibility}. Then there exist a conservative time-homogeneous $\mathcal{H}^+$-valued Markov process $X$ with transition kernels $(p_{t}(x,\cdot))_{t\geq 0}$, and constants $K,\omega \in [1,\infty)$ such that  
\begin{align}\label{eq:exp_bound}
\int_{\cHplus} \| \xi \|^2\,p_{t}(x,\D\xi)\leq K e^{\omega t} (\norm{x}^{2}+1),
\end{align}
and for all $t\geq 0$ and $u,x\in\cHplus$ we have
\begin{align}\label{eq:affine-transform-formula}
  \int_{\cHplus}\E^{-\langle \xi, u\rangle}\,p_{t}(x,\D\xi)=\E^{-\phi(t,u)-\langle x,\psi(t,u)\rangle},
\end{align}
where $\phi(\cdot,u)$ and $\psi(\cdot,u)$ are the unique solutions to the following \emph{generalized Riccati equations}
\begin{subequations}\
  \begin{empheq}[left=\empheqlbrace]{align}
    \,\frac{\partial \phi(t,u)}{\partial t}&=F(\psi(t,u)),\quad t>0,\quad \phi(0,u)=0,\label{eq:Riccati-phi}
    \\
    \,\frac{\partial \psi(t,u)}{\partial t}&=R(\psi(t,u)),\quad t>0,\quad \psi(0,u)=u,\label{eq:Riccati-psi}
  \end{empheq}
\end{subequations}
where $F\colon\cHplus\to \MR$ and $R\colon\cHplus\to\cH$ are given by
\begin{subequations}
\begin{align}
  F(u)&= \langle b,u\rangle-\int_{\cHpluso}\big(\E^{-\langle
  \xi,u\rangle}-1+\langle \chi(\xi) ,u\rangle\big)\,\dm, \label{eq:F}\\
  R(u)&= B^{*}(u)-\int_{\cHpluso}\big(\E^{-\langle
  \xi,u\rangle}-1+\langle \chi(\xi) ,
  u\rangle\big)\,\frac{\dmu}{\norm{\xi}^{2}}\,.\label{eq:R}
\end{align}
\end{subequations}
\end{theorem}

\vspace{5mm}

Define the transition semigroup $(P_t)_{t\geq 0}$ by
\[
 P_t f(x) = \int_{\mathcal{H}^+}f(\xi)p_t(x,\mathrm{d}\xi)
\]
for bounded measurable functions $f: \mathcal{H}^+ \longrightarrow \R$. Then $(P_t)_{t \geq 0}$ is a positive semigroup. Let $\rho(x) = 1 + \|x\|^2$ and define
\[
 \| f\|_{\rho} = \sup_{x \in \mathcal{H}^+} \frac{|f(x)|}{\rho(x)}.
\]
Denote by $B_{\rho}(\mathcal{H}^+)$ the Banach space of all measurable functions $f: \mathcal{H}^+ \longrightarrow \R$ for which $\|f \|_{\rho}$ is finite. Clearly $(P_t)_{t \geq 0}$ extends onto $B_{\rho}(\mathcal{H}^+)$ and satisfies for each $f \in B_{\rho}(\mathcal{H}^+)$
\begin{align}\label{eq: semigroup bound}
 |P_t f(x)| \leq \| f\|_{\rho}\int_{\mathcal{H}^+} \rho(y)p_t(x,\D y) \leq \| f\|_{\rho}(1+ K)e^{\omega t}\rho(x), \qquad x \in \mathcal{H}^+,
\end{align}
i.e. $(P_t)_{t \geq 0}$ leaves $B_{\rho}(\mathcal{H}^+)$ invariant. Let $\mathcal{H}^+_w$ be the space $\mathcal{H}^+$ equipped with the weak topology, and denote by $C_b(\mathcal{H}_w^+)$ the space of all bounded and weakly continuous functions $f: \mathcal{H}^+ \longrightarrow \R$. Finally let $\mathcal{B}_{\rho}(\mathcal{H}^+)$ be the closure of $C_b(\mathcal{H}_w^+)$ in $B_{\rho}(\mathcal{H}^+)$.
It follows from \cite{cox2020affine} that $(P_t)_{t \geq 0}$ leaves $\mathcal{B}_{\rho}(\mathcal{H}_w^+)$ invariant and satisfies $\lim\limits_{t\to 0+}P_{t}f(x)=f(x)$ for all $f\in \cB_{\rho}(\mathcal{H}_w^+)$ and $x\in \mathcal{H}_w^+$.
From \eqref{eq: semigroup bound} we then obtain
\[
 \| P_t \|_{\mathcal{L}(\mathcal{B}_{\rho}(\mathcal{H}_w^+))} \leq (1+K)e^{\omega t}, \qquad t \geq 0.
\]
Hence $(P_t)_{t \geq 0}$ is a \textit{generalized Feller semigroup} on $\mathcal{B}_{\rho}(\mathcal{H}_w^+)$ (see the Appendix~\ref{sec:gener-fell-semigr} for the definition and additional details).

\begin{remark}\label{rem:existence}
 Given the transition kernels $(p_{t}(x,\cdot))_{t\geq 0}$, the
 process $(X_{t})_{t\geq 0}$ with initial value $X_{0}=x\in\cHplus$ can be
 constructed by a version of Kolmogorov's extension theorem in
 \cite[Theorem 2.11]{CT20}. Indeed, for every $x\in\cHplus$ one can show
 the existence of a unique measure $\MP_{x}$ on $\Omega\df
 (\cH^{+})^{\MRplus}$, equipped with the $\sigma$-algebra generated by the canonical projections $X_{t}\colon \Omega \to \cH$, given by
 $X_{t}(\omega)=\omega(t)$ for $\omega\in\Omega$, are measurable. For
 $x\in\cHplus$ the probability measure $\MP_{x}$ is the distribution
 of $X$ with $\MP_{x}(X_{0}=x)=1$. We denote the expectation with
 respect to $\MP_{x}$ by $\EXspec{x}{\cdot}$.
\end{remark}

Let $(b,B,m,\mu)$ be an admissible parameter set according
to~Definition~\ref{def:admissibility} and let us denote by $(X_{t})_{t\geq 0}$
the associated affine process on $\cHplus$. 
Below we state explicit formulas for the first two moments of the process, see \cite[Proposition 4.7]{cox2020affine} for a proof.
\begin{proposition}\label{prop:CKK20}
  Let $(X_{t})_{t\geq 0}$ be the affine process associated with the admissible
  parameter set $(b,B,m,\mu)$. Then for all $v,w\in\cHplus$ the following
  formulas hold true: 
  \begin{align}\label{eq:explicit-Xv}
 \EXspec{x}{\langle X_{t},v\rangle}&=\int_{0}^{t}\langle \hat{b}, \E^{s\hat{B}}v\rangle\,\D s+\langle x, \E^{t\hat{B}}v\rangle 
  \end{align} and 
  \begin{align}\label{eq:explicit-Xv-square}
    \EXspec{x}{\langle X_{t},v\rangle\langle X_{t},w\rangle}&= \Big( 
                                                                 \int_0^t \langle\hat{b},\E^{s \hat{B}} v\rangle\,\D s 
                                                                 + 
                                                                 \big\langle x,
                                                                 \E^{t \hat{B}} v\big\rangle\Big)\Big( 
                                                                 \int_0^t \langle\hat{b},\E^{s \hat{B}} w\rangle\,\D s 
                                                                 + 
                                                                 \big\langle x,
                                                             \E^{t \hat{B}} w\big\rangle\Big)\nonumber\\
   & \quad+ \int_0^t\int_{\cHpluso}\langle \xi, \E^{s \hat{B}} v\rangle\langle \xi, \E^{s \hat{B}} w\rangle\,m(\D\xi)\,\D s
\nonumber \\ & \quad 
    +  \int_0^t \int_{0}^{s} 
        \Big\langle\hat{b},\E^{(s-u)\hat{B}}\int_{\cHpluso}\langle \xi,\E^{u \hat{B}}v\rangle\langle \xi,\E^{u \hat{B}}w\rangle\,\frac{\mu(\D\xi)}{\norm{\xi}^{2}}\Big\rangle\,\D u\, \D s
\nonumber\\ & \quad 
    + \int_{0}^{t}
        \Big\langle x,\E^{(t-s)\hat{B}}\int_{\cHpluso}\langle\xi,\E^{s
              \hat{B}}w\rangle\langle\xi,\E^{s
              \hat{B}}v\rangle\,\frac{\mu(\D\xi)}{\norm{\xi}^{2}}
              \Big\rangle\,\D s. 
\end{align}
\end{proposition}

\section{Main results}\label{sec:main-results}
Let $V_{\tau}\df(V,\tau)$ be a topological vector space and denote by $\cM(V_{\tau})$
the set of all probability measures defined on the Borel-$\sigma$-algebra
$\cB(V_{\tau})$. For the vector space $\cH$ equipped with its weak topology
$\tau_{\mathrm{w}}$ we write $\cH_{\mathrm{w}}=(\cH,\tau_{\mathrm{w}})$. Note
that the positive cone $\cHplus_{\mathrm{w}}$ is also closed in the weak
topology and moreover, the Borel-$\sigma$-algebras of the strong and
weak topology coincide, i.e. $\cB(\cHplus)=\cB(\cHplus_{\mathrm{w}})$. We say
that a measure $\nu\in\cM(\cHplus_{\tau})$ is \textit{inner-regular} (with
respect to the topology $\tau$), whenever  
\begin{align*}
 \nu(A)=\sup\set{\nu(K)\colon K\subseteq A,\, K\text{ is $\tau$-compact}}.
\end{align*}
For a sequence $(\nu_{n})_{n\in\MN}\subseteq \cM(\cHplus)$ we write
$\nu_{n}\Rightarrow \nu$ as $n\to\infty$ for the weak convergence of
$(\nu_{n})_{n\in\MN}$ to $\nu$ in the strong topology i.e. 
\begin{align*}
  \lim_{n\to\infty}\int_{\cHplus}f(\xi)\,\nu_{n}(\D\xi)=\int_{\cHplus}f(\xi)\,\nu(\D\xi)\quad\text{for
  all }f\in C_{b}(\cH). 
\end{align*}
For $\nu_{1}$, $\nu_{2}\in\cM(\cHplus)$ we call a probability measure $G$, defined on the product Borel-$\sigma$-algebra $\cB(\cHplus)\times\cB(\cHplus)$, a
\textit{coupling} of $(\nu_{1},\nu_{2})$, whenever its marginal distributions are given by $\nu_{1}$ and $\nu_2$, respectively. We denote the set of all possible couplings of $(\nu_{1},\nu_{2})$ by
$\cC(\nu_{1},\nu_{2})$. For $p\in [1,\infty)$ the
\textit{Wasserstein distance of order $p$} between $\nu_{1}\in\cM(\cHplus)$ and $\nu_{2}\in\cM(\cHplus)$ is defined as   
  \begin{align*}
   W_{p}(\nu_{1},\nu_{2})&=\left(\inf\set{\int_{\cHplus\times\cHplus}\norm{x-y}^{p}\,G(\D x,\D
                   y)\colon\,G\in\cC(\nu_{1},\nu_{2})}\right)^{1/p}.
  \end{align*}
 For an introduction to Wasserstein distances we refer to \cite[Section 6]{Vil09}.\par{}
  
  Now, let $(b,B,m,\mu)$ be an admissible parameter set and denote the spectrum of
  $\hat{B}$, the operator defined in~\eqref{eq:b-B-check}, by
  $\sigma(\hat{B})$. We introduce the following central assumption:
  \begin{assumption}{A}\label{assump:subcritical-pure-jump}
    The spectral bound $s(\hat{B})\df\sup\set{\Re(\lambda)\colon
      \lambda\in\sigma(\hat{B})}$ of $\hat{B}$ in~\eqref{eq:b-B-check} is
    strictly negative, i.e. $s(\hat{B})<0$.  
  \end{assumption}
  
  \vspace{3mm}
  
  We call an affine process $(X_{t})_{t\geq 0}$ on $\cHplus$ associated with an admissible
  parameter set $(b,B,m,\mu)$ satisfying
  Assumption~\ref{assump:subcritical-pure-jump} a \textit{subcritical} affine
    process on $\cHplus$. Recall that $\hat{B}$ is bounded and it generates
  the operator semigroup $(\E^{t\hat{B}})_{t\geq 0}$ given by
  $\E^{t\hat{B}}\df\sum_{n=0}^{\infty}\frac{(t\hat{B})^{n}}{n!}$,
where the convergence of the series is understood in the $\cL(\cH)$-norm. It
is well known that $(\E^{t\hat{B}})_{t\geq 0}$ is a uniformly continuous
semigroup, see \cite[Chapter I, Section 3]{EN00}, which
  implies that the spectral bound $s(\hat{B})$ coincides with the \emph{growth
    bound} of $(\E^{t\hat{B}})_{t\geq 0}$, see \cite[Corollary 4.2.4]{EN00}, i.e.
  \begin{align*}
  s(\hat{B})=\inf\set{w\in\MR\colon\,\exists M_{w}\geq
    1\,\text{ s.t. }\norm{\E^{t\hat{B}}}_{\cL(\cH)}\leq M_{w}\E^{w t}\;\forall
    t\geq 0}.  
  \end{align*}
  Therefore whenever Assumption~\ref{assump:subcritical-pure-jump} is satisfied, there
  exists a $M\geq 1$ and $\delta>0$ such that  
  \begin{align}\label{eq:stable-semigroup}
    \norm{\E^{t\hat{B}}}_{\cL(\cH)}\leq M \E^{-\delta t},  
  \end{align}
  in particular we could choose $\delta=-s(\hat{B})$. The following theorem is a detailed version of our main result concerning the long-time behavior of affine processes on the state-space $\cHplus$:
  \begin{theorem}\label{thm:convergence-Wasserstein}
  Let $(b,B,m,\mu)$ be an admissible parameter set such that
  Assumption~\ref{assump:subcritical-pure-jump} is satisfied. Denote the
  associated subcritical affine process on $\cHplus$ by 
  $(X_{t})_{t\geq 0}$ and its transition kernels by $(p_{t}(x,\cdot))_{t\geq
    0}$. Then the following holds true:
  \begin{theoremenum}
  \item\label{item:invariant-measure} There exists a unique invariant measure
    $\pi$ for $(p_{t}(x,\cdot))_{t\geq 0}$ and the Laplace transform of $\pi$ is given by 
    \begin{align}\label{eq:Laplace-invariant}
  \int_{\cHplus}\E^{-\langle u,x\rangle}\,\pi(\D
      x)=\exp\left(-\int_{0}^{\infty}F(\psi(s,u))\,\D s\right),\quad u\in\cHplus,  
    \end{align}
    where $F$ and $\psi(s,u)$ are as in \eqref{eq:F}
    and~\eqref{eq:Riccati-psi}. Moreover, $\pi$ is an inner-regular measure on
    $\cB(\cHplus_{\mathrm{w}})$.
  \item\label{item:convergence-rates} For $p\in
[1,2]$, $t \geq 0$ and $x\in\cHplus$ we have
\begin{align}
  W_{p}(p_{t}(x,\cdot),\pi)&\leq C_{1}\E^{-\delta
    t}\left(\norm{x}+\big(\int_{\cHplus}\norm{y}^{p}\,\pi(\D
  y)\big)^{1/p}\right)\label{eq:convergence-Wasserstein}\\
  &\quad+C_{2}\E^{-\delta/2
    t}\left(\norm{x}^{1/2}+\big(\int_{\cHplus}\norm{y}^{p/2}\,\pi(\D
    y)\big)^{1/p}\right),\label{eq:convergence-Wasserstein-2} 
\end{align}
where $C_{1}=2M$ and $C_{2}=2^{1/2}M^{3/2}\delta^{-1/2}\norm{\mu(\cHpluso)}^{1/2}$ for $M\geq 1$
and $\delta>0$ as in~\eqref{eq:stable-semigroup}. 
In particular, we have $p_{t}(x,\cdot)\Rightarrow\pi$ as $t\to\infty$.
  \end{theoremenum}  
\end{theorem}

\begin{remark}
  \begin{enumerate}
  \item[1)] For locally compact and second countable Hausdorff spaces, in
    particular for finite dimensional normed spaces, every
    probability measure defined on the Borel-$\sigma$-algebra is regular. The
    last assertion in \cref{item:invariant-measure} states that the invariant
    measure $\pi$ is an inner-regular measure on $\cB(\cHplus_\mathrm{w})$, i.e. inner-regular in the weak topology,
    albeit $\cHplus_{\mathrm{w}}$ in the infinite dimensional case is not
    locally compact. We see that the inner-regularity is a
    non-trivial property of the invariant measure and we actually use it in the
    proof of Corollary~\ref{cor:stationary} below.      
  \item[2)] For the case $p=1$ we can compare the convergence rates obtained in
    \cref{item:convergence-rates} with the ones in \cite[Theorem
    2.9]{FJKR20} for the state space $\MS_{d}^{+}$ $(d\in\MN)$ i.e. $H=\MR^{d}$ in our case. We see that instead of the square-root of the
    dimension $d\in\MN$ as it appears in the convergence rate in
    \cite[equation 2.12]{FJKR20}, we have the additional
    term~\eqref{eq:convergence-Wasserstein-2} which also converges
    exponentially fast as $t\to\infty$, but with the exponential factor
    $-\delta/2$ instead of $-\delta$. However, the convergence rates here do
    not depend on the dimension of the state-space, in particular they hold
    true in infinite dimensions.
  \end{enumerate}
\end{remark}

As a corollary from \cref{item:invariant-measure} which ensures the existence of an
invariant inner-regular measure $\pi$ on $\cB(\cHplus_{\mathrm{w}})$, we assert the
existence of a stationary process with stationary measure $\pi$. The only
assertion that is left to prove here is, that we can start an affine process
with transition kernels $p_{t}(x,\cdot)$ at distribution $\pi$ instead of $\delta_{x}$:

\begin{corollary}\label{cor:stationary}
  There exists a process $(X_{t}^{\pi})_{t\geq 0}$ on $\cHplus$ with
  transition kernels $(p_{t}(x,\cdot))_{t\geq 0}$ such that the distribution of $X^{\pi}_{t}$ equals $\pi$ for all $t\geq 0$.
\end{corollary}

Note here that the $p$-th absolute moment of $\pi$ shows up in the convergence
rate in \eqref{eq:convergence-Wasserstein}, where we implicitly assumed that
these terms are finite. That this is indeed the case is part of the next
proposition, where we also prove explicit formulas for the first two moments
of the invariant measures $\pi$.

\begin{proposition}\label{prop:limit-moment-formula}
Under the same conditions as in Theorem \ref{thm:convergence-Wasserstein} and by denoting the unique invariant measure of $(p_{t}(x,\cdot))_{t\geq 0}$
by $\pi$ we have $\int_{\mathcal{H}^+}\|y\|^2 \,\pi(\D y) < \infty$,
\begin{align}\label{eq:limit-first-moment-formula}
\lim_{t\to\infty}\EXspec{x}{X_{t}}=\int_{\cHplus}y\,\pi(\D
  y)=\int_{0}^{\infty}\E^{s\hat{B}}\Big(b+\int_{\cHplus\cap\set{\norm{\xi}>1}}\xi\,m(\D\xi)\Big)\,\D s, 
\end{align}
and
\begin{align}\label{eq:limit-second-moment-formula}
 \lim_{t\to\infty}\EXspec{x}{X_{t}\otimes X_{t}}&=\int_{\cHplus}y\otimes y\,\pi(\D y)\nonumber\\
  &=\int_{0}^{\infty}\big(\E^{s\hat{B}^{*}}\hat{b}\big)^{\otimes 2}\D
  s+\int_{0}^{\infty}\int_{\cHpluso}\big(\E^{s\hat{B}^{*}}\xi\big)^{\otimes2}\,m(\D\xi)\D
  s\nonumber\\
  &\quad+
    \int_{0}^{\infty}\int_{0}^{s}\int_{\cHpluso}\big(\E^{u\hat{B}^{*}}\xi\big)^{\otimes
    2}\langle
    \hat{b},\E^{(s-u)\hat{B}}\,\frac{\mu(\D\xi)}{\norm{\xi}^{2}}\rangle\,\D
    u\,\D s.
\end{align}
\end{proposition}

\begin{remark}
  It is well known that convergence in Wasserstein distance of order $p\in
  [1,\infty)$ implies weak convergence and the convergence of the $p$-th
  absolute moment, see \cite[Theorem 6.9]{Vil09}. However, we want to remark
  here that the proof of~\eqref{eq:limit-first-moment-formula}, given in
  Section~\ref{sec:proof-prop-refpr}, does not depend on
  the convergence in Wasserstein distance of order $p=2$ as established by
  \cref{item:convergence-rates}. Instead we solely use the generalized
  Feller property of the transition semigroups $(P_{t})_{t\geq 0}$ together with
  Proposition~\ref{prop:CKK20}.    
\end{remark}

\begin{example}[L\'evy driven Ornstein-Uhlenbeck processes]\label{ex:OU}
  Let $m$ be a L\'evy measure on $\cB(\cHpluso)$ with finite second moment and
  $b\in\cHplus$ such that \cref{item:drift} is satisfied. Let $\mu=0$ and
  $B\in\cL(\cH)$ be of the form $B(u)=Gu+uG^{*}$ for some $G\in\cL(H)$, then
  \cref{item:linear-operator} is satisfied, which can be seen from the fact
  that for every $u\in\cHplus$ we have
  $\E^{tB}u=\E^{tG} u \E^{tG^{*}}\geq_{\cHplus} 0$ for all $t\geq
  0$. Hence~\cite[Theorem 1]{LV98} implies that $B$ satisfies \cref{item:linear-operator}. Thus the tuple $(b,m,B,0)$ is an admissible
  parameter set according to Definition~\ref{def:admissibility} and the
  associated affine process $(X_{t})_{t\geq 0}$ becomes an Ornstein-Uhlenbeck
  process driven by a $\cHplus$-valued L\'evy process $(L_{t})_{t\geq 0}$ with
  characteristics $(b,0,m)$, see \cite[Lemma 2.3]{cox2021infinitedimensional}, i.e.
  \begin{align*}
    X_{t}=\E^{t G}x\E^{tG^{*}}+\int_{0}^{t}\E^{(t-s)G}\,\D
    L_{s}\E^{(t-s)G^{*}},\quad t\geq 0. 
  \end{align*}
  Since $\sigma(B)=\sigma(G)+\sigma(G)$, see \cite{Ros56}, and hence $s(B)\leq
  s(G)$, we see that whenever the spectral bound $s(G)$ of the operator $G$ is
  negative, the same holds for $s(B)$ and hence
  Assumption~\ref{assump:subcritical-pure-jump} is satisfied. This provides an explicit and simple sufficient criterion for the Ornstein-Uhlenbeck process $(X_{t})_{t\geq 0}$ to be subcritical. By Theorem~\ref{thm:convergence-Wasserstein} there exists a unique
  invariant measure $\pi$ with Laplace transform
  \begin{align*}
  \int_{\cHplus}\E^{-\langle u,x\rangle}\,\pi(\D
    x)=\exp\left(-\int_{0}^{\infty}\varphi_{L}\big( \E^{sG}u\E^{sG^{*}}\big)\,\D s\right), 
  \end{align*}
  where $\varphi_{L}\colon \cH\to \MC$ denotes the Laplace exponent of the
  L\'evy process $L$ given by
  \begin{align}\label{eq:Laplace-Levysubordinator}
  \varphi_{L}(u)= \langle b, u\rangle-\int_{\cHpluso}\big(\E^{-\langle \xi,
  u\rangle}-1+\langle \chi(\xi), u\rangle\big)\, m(\D \xi)\, ,\quad u\in \cHplus.  
  \end{align}
  Existence and uniqueness of invariant measures for Ornstein-Uhlenbeck
  processes were studied in \cite{CM87}, where a similar result follows under the
  weaker log-moment condition on the L\'evy measure $m$. Following
  Proposition~\ref{prop:limit-moment-formula} the stronger second 
  moment assumption in our case allows us to deduce explicit formulas for the
  first and second moments of $\pi$. Indeed, setting $\mu=0$ and $B(u)=Gu+uG^{*}$
  in~\eqref{eq:limit-first-moment-formula}
  and~\eqref{eq:limit-second-moment-formula} gives 
\begin{align*}
  \lim_{t\to\infty}\EXspec{x}{X_{t}}=\int_{0}^{\infty}\E^{sG}\Big(b+\int_{\cHpluso\cap\set{\norm{\xi}>1}}\xi\,m(\D\xi)\Big)\E^{sG^{*}}\,\D s, 
\end{align*}
and
\begin{align*}
  \lim_{t\to\infty}\EXspec{x}{X_{t}\otimes X_{t}}&=\int_{0}^{\infty}\big(\E^{sG}\hat{b}\E^{sG^{*}}\big)^{\otimes 2}\D
                                                   s+\int_{0}^{\infty}\int_{\cHpluso}\big(\E^{sG}\xi\E^{sG^{*}}\big)^{\otimes
                                                   2}\,m(\D\xi)\D
                                                   s. 
\end{align*}
\end{example}

\section{Applications}\label{sec:applications}

In this section we discuss applications of our results in the context of affine
stochastic covariance models in Hilbert spaces. In
Section~\ref{sec:affine-sv-models-1} we introduce an abstract Hilbert valued stochastic covariance model in the so called \textit{stationary covariance
regime} and derive the stationary affine transform formulas for examples from the literature.
In Section~\ref{sec:model-forw-curve} we then consider a concrete example of an
geometric affine stochastic covariance model describing the dynamics of commodity forward curves. In particular, we show that the implied volatility of forward-start options written on forwards modeled by this model can be related to the implied volatility of plain
vanilla options on forwards modeled under the stationary covariance regime.

\subsection{Affine SV models in the stationary variance regime}\label{sec:affine-sv-models-1}
Let $(X_{t})_{t\geq 0}$ be an affine process on $\cHplus$ with with initial value $X_{0}=x$ and associated with an admissible
parameter set $(b,B,m,\mu)$. Moreover, assume that
$(Y_{t})_{t\geq 0}$ is the unique (mild) solution to the following stochastic differential
equation on some separable Hilbert space $\left(H,(\cdot,\cdot)_{H}\right)$:
\begin{align}\label{eq:Y}
  Y_{t}=S(t)y+\int_{0}^{t}S(t-s)G(X_{s})\,\D
  s+\int_{0}^{t}S(t-s)X_{t}^{1/2}\,\D W_{t},\quad t\geq 0,  
\end{align}
where $G\colon \cH\to H$ is a continuous affine linear function,
$(W_{t})_{t\geq 0}$ is a $H$-valued Brownian motion, independent of $(X_{t})_{t\geq 0}$, with covariance operator $\cQ\in\cL_{2}(H)$, and
$(S(t))_{t\geq 0}$ is a strongly continuous semigroup on $H$ with generator $(\cA,\dom(\cA))$. We call
the joint process $(Y_{t},X_{t})_{t\geq 0}$ an \textit{affine stochastic
  covariance model} on $H$. Examples for stochastic covariance models in a Hilbert space setting can be found in~\cite{BS18, BRS18, cox2021infinitedimensional, benth2021barndorff}.
  Note that the joint process $(Y_{t},X_{t})_{t\geq 0}$ can be
considered as a stochastic process on the filtered probability space $(\Omega, \cF, \MF, \MQ_{x})\df(\Omega^{1}\times\Omega^{2}, (\cF^{1}\otimes \cF^{2}),
(\cF^{1}_{t}\otimes \cF_{t}^{2})_{t\geq 0}, \MQ \otimes \MP_{x})$,
where $(\Omega^2,\cF^2,(\mathcal{F}_t^2)_{t\ge 0}, \mathbb{P}_{x})$ denotes the
filtered probability space accommodating the affine process $(X_{t})_{t\geq 0}$, see also Remark~\ref{rem:existence}, and $(\Omega^1,\cF^1,  (\mathcal{F}_t^1)_{t\ge 0}, \mathbb{Q})$ is
another filtered probability space, that carries a $\cQ$-Wiener process $W\colon [0,\infty)\times \Omega \rightarrow H$ and the solution process $(Y_t)_{t\geq 0}$ such that $\MQ(Y_0=y)=1$. From now on we write $(Y_t^y)_{t\geq 0}$ where the superscript $y$ indicates the initial value of the process $(Y_t)_{t\geq 0}$. Moreover, we denote the expectation with respect to the product measure $\MQ_{x}$ by $\EXspec{x}{\cdot}$. 

Heuristically, the joint process
$(Y^{y}_{t},X_{t})_{t\geq 0}$ satisfies a similar transform formula for its mixed
Fourier-Laplace transform as the process $(X_{t})_{t\geq 0}$ does for the Laplace
transform in~\eqref{eq:affine-transform-intro}, see also \cite{BRS18,
  cox2021infinitedimensional}, which justifies the name affine stochastic covariance model. If moreover Assumption~\ref{assump:subcritical-pure-jump} is satisfied, then
by Theorem~\ref{thm:convergence-Wasserstein} there
exists a unique invariant measure $\pi$ for $(p_{t}(x,\cdot))_{t\geq 0}$ and by
Corollary~\ref{cor:stationary} the existence of the stationary process
$(X_{t}^{\pi})_{t\geq 0}$ is ensured. Now, if there exists a mild solution
$(\tilde{Y}_{t})_{t\geq 0}$ of~\eqref{eq:Y} for $y=0$ and where the process
$(X_{t})_{t\geq 0}$ is replaced by $(X_{t}^{\pi})_{t\geq 0}$, then we call the
joint process $(\tilde{Y}_{t},X^{\pi}_{t})_{t\geq 0}$, defined on $(\Omega, \cF, \MF,
\MQ_{\pi})$ (with $\MQ_{\pi}=\MQ\otimes \MP_{\pi}$ and the expectation with respect to $\MQ_{\pi}$ is denoted by $\EXspec{\pi}{\cdot}$) an affine stochastic covariance model on $H$ in the \textit{stationary covariance regime}. This terminology is inspired by the univariate setting in \cite[Section
3]{Kel11}. 

Below we consider a particular class of models constructed in \cite{cox2021infinitedimensional} which we call \textit{SV model with affine pure-jump variance}.
Namely, let $(Y^{y}_{t},X_{t})_{t\geq 0}$ be the process with first component given by~\eqref{eq:Y} and the second component given by the affine process $(X_{t})_{t\geq 0}$, where we assume that it has c\`adl\`ag paths and that there exists a positive and self-adjoint $D\in\cL(H)$ such that $X^{1/2}_{t}QX_{t}^{1/2}=D^{1/2}X_{t}D^{1/2}$ for all $t\geq 0$, compare with \cite[Assumption $\mathfrak{B}$ and $\mathfrak{C}$]{cox2021infinitedimensional}. It was shown that under these conditions the stochastic covariance model $(Y^{y}_{t},X_{t})_{t\geq 0}$ is well defined for every initial value $(y,x)\in H\times \cHplus$. In the following proposition we give an affine transform formula for this model in the stationary covariance regime:
\begin{proposition}\label{prop:affine-formula-stationary}
  Assume that $(Y^{y}_{t},X_{t})_{t\geq 0}$ is an affine stochastic covariance
  model satisfying the assumptions above and let $(\tilde{Y}_{t},X_{t}^{\pi})_{t\geq 0}$
  be the model in the stationary covariance regime. Then for every $T\geq 0$ and
  $u=(u_{1},u_{2})\in\I H\times \cHplus$ we have 
  \begin{align}\label{eq:characteristic-function-stationary}
    \EXspec{\pi}{\E^{ \langle \tilde{Y}_{t},u_{1}\rangle-\langle
    X^{\pi}_{t},u_{2}\rangle}}=\E^{-\Phi(t,u_{1},u_{2})-\int_{0}^{\infty}F(\psi_{2}(s,0,\psi_{2}(t,u_{1},u_{2})))\,\D
    s},\, t\in [0,T],
  \end{align}
  where $\Phi(\cdot,u_{1},u_{2})$, $\psi_{1}(\cdot,u_{1},u_{2})$ and
  $\psi_{2}(\cdot,u_{1},u_{2})$ are the unique solutions on $[0,T]$ of the following differential equations:
  \begin{subequations}
\begin{empheq}[left=\empheqlbrace]{align}
   \,\frac{\partial\Phi}{\partial t}(t,u)&=F(\psi_{2}(t,u)), &\quad\Phi(0,u)=0,\label{eq:Riccati-phi-psi-1-1}\\
    \,\psi_{1}(t,u)&=u_{1}-\I \cA^{*}\left(\I\int_{0}^{t}\psi_{1}(s,u)\D
                   s\right), &\quad\psi_{1}(0,u)=u_{1}, \label{eq:Riccati-phi-psi-1-2}\\
    \,\frac{\partial \psi_{2}}{\partial t}(t,u)&=\mathcal{R}(\psi_{1}(t,u),
    \psi_{2}(t,u)), & \quad
                                                  \psi_{2}(0,u)=u_{2},\label{eq:Riccati-phi-psi-1-3}
       \end{empheq}
     \end{subequations}
     where $F$ is as in~\eqref{eq:F}, $\mathcal{R}(h,u)\df
     R(u)-\tfrac{1}{2}D^{1/2}h\otimes D^{1/2}h$ with $R$ as in~\eqref{eq:R}, and $(\mathcal{A}^*, D(\mathcal{A}^*))$ denotes the adjoint operator of the generator $(\mathcal{A}, D(\mathcal{A}))$ of $(S(t))_{t \geq 0}$.
\end{proposition}
\begin{proof}
  Let $T\geq 0$ and let $(Y^{y}_{t})_{t\in [0,T]}$ be the mild solution
  to~\eqref{eq:Y} on $[0,T]$ satisfying the assumptions above. From
  \cite[Theorem 3.3]{cox2021infinitedimensional} we recall the following affine
  transform formula for the mixed Fourier-Laplace transform of
  $(Y^{y}_{t},X_{t})$ for $t\in[0,T]$ and $u=(u_{1},u_{2})\in \I H\times\cHplus$:   
  \begin{align}\label{eq:affine-formula-joint}
    \EXspec{x}{\E^{\langle Y^{y}_{t},u_{1}\rangle-\langle
    X_{t},u_{2}\rangle}}=\E^{-\Phi(t,u_{1},u_{2})-\langle x,
    \psi_{2}(t,u_{1},u_{2})\rangle},\quad x\in\cHplus,
  \end{align}
  where $\Phi(\cdot,u_1,u_2)$ and $\psi_2(\cdot,u_1,u_2)$ are the unique strong solutions to~\eqref{eq:Riccati-phi-psi-1-1} and \eqref{eq:Riccati-phi-psi-1-3}, respectively, and $\psi_1(\cdot,u_1,u_2)$ is the unique mild solution of~\eqref{eq:Riccati-phi-psi-1-2}.
  Note that for every $\cF^2$-measurable and bounded function $f$ we have
  \begin{align*}
   \int_{\Omega_{2}}
    f(\omega_{2})\,\D\MP_{x}(\omega_{2})=\int_{\cHplus}\left(\int_{\Omega_{2}}
    f(\omega_{2})\,\D\MP_{x}(\omega_{2})\right)\pi(\D  x).
  \end{align*}
  From this and ~\eqref{eq:affine-formula-joint} we conclude
  \begin{align}
   \EXspec{\pi}{\E^{ \langle \tilde{Y}_{t},u_{1}\rangle-\langle
    X^{\pi}_{t},u_{2}\rangle}}&=\int_{\Omega_{2}}\Big( \int_{\Omega_{1}}\E^{
                          \langle \tilde{Y}_{t}(\omega_{1},\omega_{2}),u_{1}\rangle_{H}-\langle
                          X^{\pi}_{t}(\omega_{2}), u_{2}\rangle}\,\D
                                \MQ(\omega_{1})\Big)\,\D\MP_{\pi}(\omega_{2})\nonumber\\
                        &=\int_{\cHplus}\EXspec{x}{\E^{\langle
                          Y_{t},u_{1}\rangle_{H}-\langle
                          X_{t},u_{2}\rangle}}\,\pi(\D x)\nonumber\\
                        &=\int_{\cHplus}\E^{-\Phi(t,u_{1},u_{2})-\langle
                          x,\psi_{2}(t,u_{1},u_{2})\rangle}\,\pi(\D x).\label{eq:affine-formula-stationary-1}
  \end{align}
  From~\cref{item:invariant-measure} it then follows that 
  \begin{align}\label{eq:affine-formula-stationary-2}
  \int_{\cHplus}\E^{-\langle
    x,\psi_{2}(t,u_{1},u_{2})\rangle}\,\pi(\D
    x)=\exp\left(-\int_{0}^{\infty}F\big(\psi(s,\psi_{2}(t,u_{1},u_{2}))\big)\,\D
    s\right),
  \end{align}
  where $\psi(\cdot,u)$ is given
  by~\eqref{eq:Riccati-psi}. From~\eqref{eq:Riccati-phi-psi-1-3} and the
  definition of $\mathcal{R}$ we see that $\psi_{2}(t,0,u_{2})=\psi(t,u_{2})$
  for every $u_{2}\in\cHplus$, hence multiplying both sides
  of~\eqref{eq:affine-formula-stationary-2} by $\E^{-\Phi(t,u_{1},u_{2})}$
  together with~\eqref{eq:affine-formula-stationary-1} yields the desired
  formula~\eqref{eq:characteristic-function-stationary}.     
\end{proof}

In a more specific setting, the authors in~\cite{BRS18} proposed an \textit{operator Barndorff--Nielsen-Shepard} (BNS) stochastic covariance model which can be described by the following pair of Hilbert valued SDEs:
\begin{align}\label{BNS}
  \begin{cases} \D Y_{t} &= \cA(Y_{t})\D t+X_{t}^{1/2}\D W_{t},\quad
    Y_{0}=y\in H,\\
  \D X_{t}&= B(X_{t})\D t+ \D L_{t},\quad X_{0}=x\in\cHplus,
  \end{cases}
\end{align}
where $B$ and $(L_{t})_{t\geq 0}$ are as in Example~\ref{ex:OU}.
Note that $(Y^{y}_t)_{t\geq 0}$ is as in~\eqref{eq:Y} with $G=0$ written in the differential form, where $(S(t))_{t\geq 0}$ is the strongly continuous semigroup generated by $(\cA,\dom(\cA))$. For the $\cHplus$-valued Ornstein-Uhlenbeck process $(X_t)_{t\geq 0}$ we already showed the existence of a unique invariant measure $\pi$ of $(X_{t})_{t\geq 0}$ in Example~\ref{ex:OU}. Hence we may consider the \textit{operator BNS model in the stationary covariance regime} and denote it by $(\tilde{Y}_{t}, X^{\pi}_{t})_{t\geq 0}$. 
In \cite{cox2021infinitedimensional} it was shown that the operator valued BNS model is a particular case of the SV models with affine pure-jump variance as introduced above. Hence we obtain from Proposition~\ref{prop:affine-formula-stationary} applied to this particular case
\begin{align*}
 \EXspec{\pi}{\E^{\langle \tilde{Y}_{t},u_{1}\rangle_{H}-\langle X^{\pi}_{t},
  u_{2}\rangle}}=\exp\left(-\int_{0}^{t}\varphi_{L}\big(\psi(s,u_{1},u_{2})\big)\D s-\int_{0}^{\infty}\varphi_{L}\big(e^{s B^*}u_2\big)\D s\right),
\end{align*}
for every $(u_{1},u_2)\in \I H\times\cHplus$, where $\varphi_{L}$ is given by \eqref{eq:Laplace-Levysubordinator} and $\psi(t,u_{1},u_{2})$ is explicitly known as
\begin{align*}
 \psi(t,u_{1},u_{2})=\E^{sB^{*}}u_{2}+\tfrac{1}{2}\int_{0}^{s}\e^{(s-\tau)B^*}(D^{1/2}S^{*}(\tau)
  u_{1})^{\otimes 2}\,\D\tau.
\end{align*}

\subsection{Forward curve dynamics and forward-start options on
  forwards}\label{sec:model-forw-curve} 

In this section we consider a concrete example of a geometric affine stochastic
covariance model on a Hilbert space which describes the dynamics of forward
curves in fixed income or commodity markets. Then, in Proposition~\ref{prop:implied-vol-forward-start} we study the
long-time behavior of the forward implied volatility smile in this model. First, we recall from~\cite{Fil01, FTT10} the class
of Hilbert spaces consisting of forward curves: for $\beta>0$ we denote by $H_{\beta}$ the space of all absolutely continuous functions $f\colon \MRplus\to\MR$ such that
$\norm{f}_{\beta}\df\big(|f(0)|^{2}+\int_{\MRplus}\E^{\beta |x|}|f'(x)|^{2}\,\D
x\big)^{1/2}<\infty$. The space $H_{\beta}$ is a separable Hilbert space when equipped with the inner product
\begin{align*}
\langle f,g \rangle_{\beta}=f(0)g(0)+\int_{\MRplus}\E^{\beta|x|}f'(x)g'(x)\,\D x.
\end{align*}
Moreover, we note that the left-shift semigroup, denoted by $(S(t))_{t\geq 0}$, is strongly continuous on $H_{\beta}$
and for every $t\in \MRplus$ the point evaluation maps $\delta_{t}\colon
H_{\beta}\to\MR$ are continuous linear functionals, see \cite[Theorem
2.1]{FTT10}. Throughout this section we shall identify the point evaluation
maps $\delta_{t}\,(t\in\MRplus)$ with an element $u_{t}\in H_{\beta}$ such that
$\langle f,u_{t}\rangle_{\beta}=\delta_{t}(f)$.

Let $0\leq T\leq \hat{T}$ and denote by
$F(T,\hat{T})$ the forward price at time $T$ with delivery/maturity date
$\hat{T}$, e.g. $F(T,\hat{T})$ denotes the price at time $T$ for the delivery
of one unit of an underlying spot commodity at time $\hat{T}$. We follow the
HJMM approach and model the dynamics of $(F(T,\hat{T}))_{T\leq \hat{T}}$ directly by means of a
(transformed) stochastic differential equation on $H_{\beta}$. Namely, we let
$(Y^{y}_{t},X_{t})_{t\geq 0}$ be an affine stochastic covariance model as
in~\eqref{eq:Y} on the Hilbert space $H_{\beta}$ and where $(S(t))_{t\geq 0}$
is the left-shift semigroup on $H_\beta$. Then for $0\leq T\leq \hat{T}$ we set 
\begin{align}\label{eq:forward-model}
F(T,\hat{T})\df\delta_{\hat{T}-T}(\exp(Y^{y}_{T}))=\exp(\langle Y^{y}_{T},u_{\hat{T}-T}\rangle_{\beta}).
\end{align}

A \emph{geometric forward curve model} of the type in~\eqref{eq:forward-model}
was proposed in~\cite{benth2021barndorff} to model the dynamics of forward
curve dynamics in commodity markets. In their case, the underlying stochastic
covariance model $(Y^{y}_t,X_t)_{t\geq 0}$ is the Hilbert valued BNS model
from \eqref{BNS} with an additional leverage term. The consideration of a
geometric model has the advantage of producing positive forward curves which
is a crucial characteristic in many forward markets. Here, we extend the
geometric Hilbert valued BNS model to the larger class of \textit{SV models
  with affine pure-jump variance} as introduced in
Section~\ref{sec:affine-sv-models-1}. Moreover, we make the assumption that we
already model under a risk-neutral measure $\tilde{\MQ}$, i.e. we assume that
$(F(T,\hat{T}))_{T\leq \hat{T}}$ is a $\tilde{\MQ}$-martingale for all
$\hat{T}\geq 0$. We refer to \cite[Proposition 6.8]{benth2021barndorff} for a sufficient condition on the function $G$ that ensures the existence of a risk-neutral measure. In the following we focus on forwards in commodity markets, see also~\cite{BK14}.\par{} 

A \textit{forward-start option} with forward-start date $\tau\geq 0$, forward
maturity $T$ and strike $\E^{K}$ written on a forward with maturity date $\hat{T}$ is defined as an
European option with pay-off at time $\tau+T$ given by 
\begin{align}\label{eq:forward-start-pay-off}
 \left(\frac{F(\tau+T,\tau+\hat{T})}{F(\tau,\tau+\hat{T})}-\E^{K}\right)^{+}. 
\end{align}
A forward-start option is a contract on the relative price difference of
a forward contract at two times, $\tau$ and $\tau+T$, in the future. In practice, it is used to price future volatility
of the underlying asset. Forward-start options are very common in commodity
forward markets and more complex derivatives such as \emph{Cliquet options}
are building up on these see, e.g. ~\cite{Cro08}. Forward-start options on
stocks are discussed in, e.g.~\cite{JR13, JR15, KN05, Kel11}. Here we restrict
ourselves to the ratio-type pay-off functions
in~\eqref{eq:forward-start-pay-off}, but similar results can be obtained for
the difference-type pay-offs, i.e. $(F(\tau+T,\tau+\hat{T})-K F(\tau,\tau+\hat{T}))^{+}$, see also~\cite{KN05}.\par{}

We now define the \emph{implied forward volatility smile} of the model~\eqref{eq:forward-model}. First, let us denote the price of a forward-start option with pay-off~\eqref{eq:forward-start-pay-off} by $C_{\mathrm{fwd}}(\tau,T,\hat{T},K)$. Then, as a reference model for the forward prices $F(T,\hat{T})$ we take \emph{Black's model}, see~\cite{Bla76}, and denote the forward prices within this model by $F^{\mathrm{B}}(T,\hat{T})$. We assume that the following spot-forward relation holds:
\begin{align}\label{eq:Black-76}
F^{\mathrm{B}}(T,\hat{T})=s_{T}\E^{r(\hat{T}-T)},\quad 0\leq T\leq\hat{T},
\end{align}
where $r\geq 0$ denotes the risk-free interest rate and $(s_{t})_{t\geq 0}$
denotes the spot price process of the underlying commodity, which is given by
a geometric Brownian motion with volatility parameter $\sigma$. We denote by
$C^{\mathrm{B}}_{\mathrm{fws}}(\tau,T,\hat{T},K,\sigma)$ the price of a
forward-start option with identical pay-off function as
in~\eqref{eq:forward-start-pay-off} in Black's model and define the
\emph{implied forward volatility smile} $\sigma(\tau,T,\hat{T},K)$ as the
unique solution to
$C^{\mathrm{B}}_{\mathrm{fws}}(\tau,T,\hat{T},K,\sigma(\tau,T,\hat{T},K))=C_{\mathrm{fwd}}(\tau,T,\hat{T},K)$. In
the following proposition we show that $\sigma(\tau,T,\hat{T},K)$ exists for all $\tau,K \geq 0$ and study its long-time behavior as $\tau\to\infty$:

\begin{proposition}\label{prop:implied-vol-forward-start}
Let $0\leq T\leq \hat{T}$ and denote by $F(T,\hat{T})$ the forward price at
time $T$ with maturity date $\hat{T}$ given by~\eqref{eq:forward-model}, where
$(Y^{y}_{t},X_{t})_{t\geq 0}$ is an affine stochastic covariance model on
$H_{\beta}$ as defined in Section~\ref{sec:affine-sv-models-1} with $(S(t))_{t\geq 0}$ the left-shift semigroup on $H_\beta$. Moreover, let $(\tilde{Y}_{t},X_{t}^{\pi})_{t\geq 0}$ be the model in the stationary covariance regime and define $\tilde{F}(T,\hat{T})\df\exp\big(\langle \tilde{Y}_{T},u_{\hat{T}-T}\rangle_\beta\big)$. Suppose we model directly under the pricing measure $\tilde{Q}$ such that $(F(T,\hat{T}))_{T\leq\hat{T}}$ is a $\tilde{Q}$-martingale for all $\hat{T}\geq 0$. Then for all $\tau,K\geq 0$ the implied forward volatility smile $\sigma(\tau,T,\hat{T},K)$ exists and we have
  \begin{align}\label{eq:implied-vol-forward-start}
    \lim_{\tau\to\infty}\sigma(\tau,T,\hat{T},K)=\tilde{\sigma}(T,\hat{T},K), 
  \end{align}
  where $\tilde{\sigma}(T,\hat{T},K)$ denotes the implied volatility of a European call option with pay-off function
  $\big(\tilde{F}(T,\hat{T})-K\big)^{+}$.
\end{proposition}
\begin{proof}
First, we show a certain relation between the price of a European call option
$C^{\mathrm{B}}\left(T,\hat{T},K, \sigma \right)$, where $0\leq T\leq \hat{T}$
and $K\geq 0$, and the forward-start call option
$C_{\mathrm{fws}}^{\mathrm{B}}\left(\tau,T,\hat{T},K, \sigma\right)$ with
forward start date $\tau\geq 0$ in Black's model. Namely, let $\MQ$ denote the
unique risk-neutral measure in Black's model and recall $C_{\mathrm{fws}}^{\mathrm{B}}\left(\tau,T,K, \sigma\right)$ the price of a forward-start option at time zero with forward-start date $\tau$ written on the forward with maturity $\hat{T}$. Inserting ~\eqref{eq:Black-76} into the pay-off function and by risk-neutral pricing we have
    $C_{\mathrm{fws}}^{\mathrm{B}}\left(\tau,T,\hat{T},K, \sigma\right)=\E^{-r(\tau+T)}\EXspec{\MQ}{\Big(\frac{s_{\tau+T}\E^{r(\hat{T}-T)}}{s_{\tau}\E^{r\hat{T}}}-\E^{K}\Big)^{+}}.$
  It is known that in the Black-Scholes model the forward-start call option and the European call option satisfy:
  $\E^{-r(\tau+T)}\EXspec{\MQ}{\big(\frac{s_{\tau+T}}{s_{\tau}}-K\big)^{+}}=\E^{-r(\tau+T)}\EXspec{\MQ}{\big(s_{T}-K\big)^{+}}$,
  see also~\cite{Kel11}, and hence
  \begin{align*}
    C_{\mathrm{fws}}^{\mathrm{B}}\left(\tau,T,\hat{T},K,\sigma\right)&=\E^{-r(\tau+T)}\E^{-rT}\EXspec{\MQ}{\big(\frac{s_{\tau+T}}{s_{\tau}}-\E^{K}\big)^{+}}\\
                                          &=\E^{-r(\tau+T)}\E^{-rT}\EXspec{\MQ}{\big(s_{T}-\E^{K'}\big)^{+}}\\
                                           &=\E^{-r(\tau+T)}C^{\mathrm{BS}}\left(T,K',\sigma\right),
  \end{align*}
  where $K'=K+rT$ and by the superscript $\mathrm{BS}$ we indicate that the underlying model for $(s_t)_{t\geq 0}$ is the Black-Scholes model.
  From this and the definition of the implied forward volatility smile $\sigma(\tau,T,\hat{T},K)$ we have
\begin{align}
  C^{\mathrm{BS}}\left(T,K',\sigma(\tau,T,\hat{T},K)\right)&=\E^{r(\tau+T)}C_{\mathrm{fws}}^{\mathrm{B}}\left(\tau,T,\hat{T},K,\sigma(\tau,T,\hat{T},K)\right)\nonumber\\
                                                &=\E^{r(\tau+T)}C_{\mathrm{fws}}\left(\tau,T,\hat{T},K\right).    \label{eq:implied-volatility-1}
\end{align}
Next, we compute the right-hand side
of~\eqref{eq:implied-volatility-1}. Recall that for every $t\in\MR$ and $f\in
H_{\beta}$ we use the identification $\delta_{t}(f)=\langle f, u_{t}\rangle$ for the evaluation
functional with $ u_{t}\in H_{\beta}$. Moreover, we denote the expectation with respect to the pricing measure $\tilde{\MQ}$ by $\EXspec{\tilde{\MQ}}{\cdot}$ (here we suppress the initial value $x$ compared to $\MQ_{x}$ above). The payoff function of the forward-start option is given
by~\eqref{eq:forward-start-pay-off}, hence by risk-neutral pricing and inserting our model~\eqref{eq:forward-model} we have
\begin{align}\label{eq:new-eq-1}
  C_{\mathrm{fws}}\left(\tau,T,\hat{T},K\right)
  &=\E^{-r(\tau+T)}\EXspec{\tilde{\MQ}}{\left(\frac{F(\tau+T,\tau+\hat{T})}{F(\tau,\tau+\hat{T})}-\E^{K}\right)^{+}}\nonumber\\ 
  &=\E^{-r(\tau+T)}\EXspec{\tilde{\MQ}}{\left(\E^{\langle Y^{y}_{\tau+T},u_{\hat{T}-T}\rangle_\beta-\langle
    Y^{y}_{\tau},u_{\hat{T}}\rangle_\beta}-\E^{K}\right)^{+}}.
\end{align}
Note that by definition of $Y_{\tau}$ we have
\begin{align*}
\langle S(T)Y_{\tau},u_{\hat{T}-T}\rangle_{\beta}&=\langle
  S(T+\tau)y,u_{\hat{T}-T}\rangle_\beta+\langle \int_{0}^{\tau}S(T+\tau-s)G(X_s)
                                                      \D
                                                      s,u_{\hat{T}-T}\rangle_\beta\\
                                                    &\quad+\langle
                                        \int_{0}^{\tau}S(T+\tau-s)X_{s}^{1/2}\, \D
                                         W_{s},u_{\hat{T}-T}\rangle_\beta,
\end{align*}
and the left-shift $S(T)$ satisfies $\langle
S(T)Y^{y}_{\tau},u_{\hat{T}-T}\rangle_\beta =\langle Y_{\tau}^{y},
u_{\hat{T}}\rangle_\beta$, thus we conclude that
\begin{align}\label{eq:new-eq-2}
  \langle Y^{y}_{\tau+T},u_{\hat{T}-T}\rangle_\beta-\langle
  Y^{y}_{\tau},u_{\hat{T}}\rangle_\beta&=\langle \int_{\tau}^{T+\tau}S(T+\tau-s)G(X_s)\D s,u_{\hat{T}-T}\rangle_\beta\nonumber\\
                                        &\quad+\langle \int_{\tau}^{T+\tau}S(T+\tau-s)X_{s}^{1/2}\,\D W_{s},u_{\hat{T}-T}\rangle_\beta.
\end{align}
By the independent increments property and the Markov property of $(X_t)_{t\geq 0}$ the sum of the integrals
inside the inner products on the right-hand side of~\eqref{eq:new-eq-2} has
the same distribution as $Y^{0}_{T}=\int_{0}^{T}S(T-s)G(X_{\tau+s})\, \D
s+\int_{0}^{T}S(T-s)X_{\tau+s}^{1/2}\, \D W_{s}$. Hence for the expectation
on the right-hand side in~\eqref{eq:new-eq-1} we obtain
\begin{align*}
  \EXspec{\tilde{\MQ}}{\left(\E^{\langle Y^{y}_{\tau+T},u_{\hat{T}-T}\rangle_\beta-\langle
    Y^{y}_{\tau},u_{\hat{T}}\rangle_\beta}-\E^{K}\right)^{+}}=\EXspec{\tilde{\MQ}}{\CEX{\left(\E^{\langle
    Y^{0}_{T},u_{\hat{T}-T}\rangle_\beta}-\E^{K}\right)^{+}}{X_{\tau}}},
\end{align*}
and thus we conclude that the left-hand side of~\eqref{eq:implied-volatility-1} is given by
\begin{align*}
C^{\mathrm{BS}}\left(T,K',\sigma(\tau,T,\hat{T},K)\right)=\EXspec{\tilde{\MQ}}{\CEX{(\E^{\langle
  Y^{0}_{T},u_{\hat{T}-T}\rangle_\beta}-\E^{K})^{+}}{X_{\tau}}}.  
\end{align*}
Now taking the limit $\tau\to\infty$ and since $\tilde{F}(T,\hat{T})=\E^{\langle \tilde{Y}_{T},u_{\hat{T}-T}\rangle_\beta}$ we obtain
\begin{align}\label{eq:implied-volatility-2}
  \lim_{\tau\to\infty}C^{\mathrm{BS}}\left(T,K',\sigma(\tau,T,\hat{T},K)\right) &=\EXspec{\tilde{\MQ}}{(\E^{\langle \tilde{Y}_{T},u_{\hat{T}-T}\rangle_\beta}-\E^{K})^{+}}\nonumber\\
                                                                       &=\EXspec{\tilde{\MQ}}{\big(\tilde{F}(T,\hat{T})-\E^{K}\big)^{+}}.
\end{align}
The term $\EXspec{\tilde{\MQ}}{\big(\tilde{F}(T,\hat{T})-\E^{K}\big)^{+}}$ on the right-hand side of~\eqref{eq:implied-volatility-2} is precisely the price
of an European call option remunerated by $\E^{rT}$ and we have
\begin{align*}
 \lim_{\tau\to\infty}C^{\mathrm{BS}}\left(T,K',\sigma(\tau,T,\hat{T},K)\right)=C^{\mathrm{BS}}\left(T,K',\lim_{\tau\to\infty}\sigma(\tau,T,\hat{T},K)\right),  
\end{align*}
from which we conclude~\eqref{eq:implied-vol-forward-start}, since equation~\eqref{eq:implied-volatility-2}
has a unique solution in terms of the Black-Scholes implied volatility. 
\end{proof}

\begin{remark}
By Proposition~\ref{prop:implied-vol-forward-start} we can approximate
the option prices of forward-start option for large forward dates $\tau$ by
European plain-vanilla options modeled in the stationary covariance regime.
The relevance of this result becomes evident when noting that option pricing in affine
stochastic covariance models (in particular in the stationary covariance regime) can be
conducted with a reasonable computational effort. Indeed, since an affine transform formula is provided by
Proposition~\ref{prop:affine-formula-stationary}, the option prices can be
computed quasi-explicitly via Fourier inversion, see e.g. \cite{CM99, KMKV11,
  HKRS17}. In an ongoing work we study option pricing in general affine
stochastic covariance models on Hilbert spaces, derive the quasi-explicit
option pricing formulas for popular option types, e.g. European call or \emph{Spread options} in commodity forward markets, and conduct a numerical analysis of the pricing procedure.
\end{remark}

\section{Proof of the main results}\label{sec:proof-main-results}

Throughout this section we assume that $(b,B,m,\mu)$ is an admissible
parameter set according to Definition~\ref{def:admissibility}. We denote the
unique subcritical affine process associated with $(b,B,m,\mu)$, through Theorem~\ref{thm:CKK20}, by
$(X_{t})_{t\geq 0}$ and its family of transition kernels by
$(p_{t}(x,\cdot))_{t\geq 0}$. We set $P_{t}f\df\int_{\cH}f(\xi)\,p_{t}(\cdot,\D\xi)$ for
all measurable functions $f$ such that the integral exists. Recall from
\cite{cox2020affine} that the transition semigroup $(P_{t})_{t\geq 0}$ is a
generalized Feller semigroup on the space $\Bsq$, see also Appendix~\ref{sec:gener-fell-semigr}. 

\subsection{Some properties of the generalized Riccati equations
 \texorpdfstring{\eqref{eq:Riccati-phi}-\eqref{eq:Riccati-psi}}{the
   generalized Riccati equations}}\label{sec:gener-ricc-equat}

In this section we consider the long-time behavior of the solutions
$\phi(\cdot,u)$ and $\psi(\cdot,u)$ of the generalized Riccati equations
in~\eqref{eq:Riccati-phi}-\eqref{eq:Riccati-psi}. We recall from \cite[Section
3]{cox2020affine} that for every $u\in\cHplus$ there exists a unique and
global solution $\psi(\cdot,u)\in C^{1}(\MRplus,\cHplus)$
to~\eqref{eq:Riccati-psi}. Given $\psi(\cdot,u)$ we solve the first
equation~\eqref{eq:Riccati-phi} by mere integration and obtain
$\phi(\cdot,u)\in C^{1}(\MRplus,\MRplus)$ given by
$\phi(t,u)=\int_{0}^{t}F(\psi(s,u))\,\D s$. This means that we can write the affine transform
formula~\eqref{eq:affine-transform-formula} as
\begin{align*}
  \int_{\cHplus}\E^{-\langle\xi,u\rangle}\,p_{t}(x,\D\xi)=\exp
  \left(
  -\int_{0}^{t}F(\psi(s,u))\,\D s-\langle x, \psi(t,u)\rangle
  \right).
\end{align*}
Moreover, we recall that the unique solution $\psi(\cdot,u)$
to~\eqref{eq:Riccati-psi} satisfies the flow equation:
\begin{align}\label{eq:semi-flow}
\psi(t+s,u)=\psi(t,\psi(s,u)).  
\end{align}

In the next lemma we show that $F$ and $R$ are continuous functions on
$\cHplus$ and grow at most quadratically:

\begin{lemma}\label{lemma:RFcontinuous}
Let $(b,B,m,\mu)$ be an admissible parameter set according to Definition~\ref{def:admissibility} 
and let $F$ and $R$ be given by~\eqref{eq:F} and~\eqref{eq:R}, respectively. 
Then $F$ and $R$ are continuous on $\cHplus$ and for all $u\in \cH^+$ we have   
\begin{align}\label{eq:Fquadratic}
 | F(u)|\leq \left( \norm{b}+ \int_{\cH^+\setminus \{0\}}\norm{\xi}^2\,m(\D\xi)\right)\big(\norm{u}+\|u\|^2\big)\,,
\end{align}
and 
\begin{align}\label{eq:Rquadratic}
 \norm{R(u)}\leq \big( \norm{B}_{\cL(\cH)}+\norm{\mu(\cH^+\setminus \{0\})}\big)\left(\norm{u} + \|u\|^2\right) \,. 
\end{align} 
\end{lemma}
\begin{proof}  
Note that for all $\xi,u \in \cH^{+}$ we have
\begin{align}\label{eq:exp_est} 
 \left|\e^{ - \langle \xi, u \rangle}-1+ \langle \chi(\xi), u \rangle \right| 
  &\leq \frac{1}{2} | \langle \xi, u \rangle|^2 \one_{\{ \| \xi \| \leq 1\}} + | \langle \xi, u \rangle| \one_{\{ \| \xi \| > 1\}}
 \\ &\leq \frac{1}{2} \| \xi \|^2 \| u \|^2 \one_{\{ \| \xi \| \leq 1\}} + \| \xi\| \|u\| \one_{\{\| \xi\|>1\}}, \notag
\end{align}
from which~\eqref{eq:Fquadratic}, \eqref{eq:Rquadratic}, and the continuity of $F,R$ readily follows (by dominated convergence).
\end{proof}

Assumption~\ref{assump:subcritical-pure-jump} implies that the semigroup
$(\E^{t\hat{B}})_{t\geq 0}$ satisfies~\eqref{eq:stable-semigroup}, that is
$(\E^{t\hat{B}})_{t\geq 0}$ is uniformly exponential stable. This has
the following consequence on the solution $\psi(\cdot,u)$ of the generalized
Riccati equation~\eqref{eq:Riccati-psi}:     

\begin{lemma}\label{lem:norm-bound-psi}
Let $(b,B,m,\mu)$ be an admissible parameter set according to
Definition~\ref{def:admissibility} and for $u\in\cHplus$ let
$\psi(\cdot,u)$ be the unique solution to~\eqref{eq:Riccati-psi}. Then
\begin{align}\label{eq:norm-bound-psi}
\norm{\psi(t,u)}\leq \norm{\E^{t\hat{B}}}_{\cL(\cH)}\norm{u},\quad \forall
  t\geq 0.  
\end{align}
If moreover Assumption~\ref{assump:subcritical-pure-jump} is satisfied, then
$\lim_{t\to\infty}\psi(t,u)=0$.
\end{lemma}
\begin{proof}
  First note that whenever $(b,B,m,\mu)$ is an admissible parameter set
  according to Definition~\ref{def:admissibility}, then so is $(0,B,0,\mu)$. Therefore the existence of an affine Markov process $(Y_{t})_{t\geq 0}$
  associated to $(0,B,0,\mu)$ and initial value $Y_{0}=x$ is guaranteed by
  Theorem~\ref{thm:CKK20}. Note that the unique solution $\psi(\cdot,u)$
  to~\eqref{eq:Riccati-psi} is $\cHplus$-valued. Let $u\in\cHplus$ and note
  that due to the convexity of the exponential function and Jensen's inequality we have $\E^{-\EXspec{x}{\langle u, Y_{t} \rangle}}\leq \EXspec{x}{\E^{-\langle u,
    Y_{t}\rangle}} = \exp(-\langle \psi(t,u),x\rangle)$. Hence we find that for all $u,x\in\cHplus$:
  \begin{align*}
   \langle \psi(t,u),x\rangle \leq \EXspec{x}{\langle u, Y_{t}\rangle}
   &= \langle x, \E^{t\hat{B}}u\rangle
   \\ &\leq \norm{x}\norm{u}\norm{\E^{t\hat{B}}}_{\cL(\cH)}.
  \end{align*}
  For fixed $t\geq 0$ and $u\in\cHplus$ we choose $x=\psi(t,u) \in \mathcal{H}^+$ and obtain $\norm{\psi(t,u)}^{2}\leq
  \norm{\psi(t,u)}\norm{u} \norm{\E^{t \hat{B}}}_{\cL(\cH)}$, which proves the
  first statement. If Assumption~\ref{assump:subcritical-pure-jump} is
  satisfied, then from~\eqref{eq:norm-bound-psi}
  and~\eqref{eq:stable-semigroup} it follows that $\norm{\psi(t,u)}\leq
  M\E^{-\delta t}\norm{u}$ and hence $\lim_{t\to\infty}\psi(t,u)=0$.
\end{proof}
\subsection{Invariant measure for affine processes on
  \texorpdfstring{$\cHplus$}{the cone of positive self-adjoint Hilbert-Schmidt
    operators}}\label{sec:invar-distr-affine}
For two measures $\nu_{1},\nu_{2}\in\cM(\cHplus)$ we denote the convolution of
$\nu_{1}$ and $\nu_{2}$ by $\nu_{1}*\nu_{2}$. In the following lemma we give
an important convolution property of the transition kernels $p_{t}(x,\cdot)$.

\begin{lemma}\label{lem:convolution-homogeneous}
Let $(Y_{t})_{t\geq 0}$ be the unique affine process associated with the
admissible parameter set $(0,B,0,\mu)$ and denote its transition kernels by
$(q_{t}(x,\cdot))_{t\geq 0}$. Then for every $t\geq 0$ and $x\in\cHplus$ we have
\begin{align}\label{eq:convolution-homogeneous}
p_{t}(x,\cdot)=p_{t}(0,\cdot)*q_{t}(x,\cdot).  
\end{align}
\end{lemma}
\begin{proof}
  Since $b=0$ and $m=0$ the function $F$ in \eqref{eq:Riccati-phi} vanishes,
  see also~\eqref{eq:F}, and thus $\phi(t,u)=0$ for all 
  $t\geq 0$. Hence for every $t\geq 0$ the affine-transform
  formula~\eqref{eq:affine-transform-formula} for $Y_{t}$ takes the form 
  \begin{align}
   \int_{\cHplus}\E^{-\langle u, \xi\rangle}\,q_{t}(x,\D\xi)=\exp\big(-\langle
    \psi(t,u),x\rangle\big),\quad \text{for }u\in\cHplus.    
  \end{align}
  Now, let $(X_{t})_{t\geq 0}$ denote the unique affine process associated with
  the admissible parameter set $(b,B,m,\mu)$ and denote its transition kernels
  by $p_{t}(x,\cdot)$. Let $t\geq 0$ arbitrary and $u\in\cHplus$, then
  \begin{align*}
   \int_{\cHplus}\E^{-\langle u,\xi
    \rangle}\,p_{t}(0,\cdot)*q_{t}(x,\cdot)(\D\xi)&=\int_{\cHplus}\Big(\int_{\cHplus}\E^{-\langle
                                                    u,\xi_{1}+\xi_{2}\rangle}\,p_{t}(0,\D\xi_{1})\Big)\,q_{t}(x,\D\xi_{2})\\
                                                  &=\E^{-\phi(t,u)}\int_{\cHplus}\E^{-\langle
                                                    u,\xi_{2}\rangle}\,q_{t}(x,\D\xi_{2})\\
                                                  &=\E^{-\phi(t,u)}\E^{-\langle
                                                    \psi(t,u),x\rangle},
  \end{align*}
  which completes the proof thanks to~\eqref{eq:affine-transform-formula}
  and the fact that the functions $x\mapsto \E^{-\langle u,x\rangle}$
  characterize measures, see \cite[Lemma A.1]{cox2021infinitedimensional}.
\end{proof}

In the next proposition we show that the Laplace transform of a subcritical
affine process converges pointwise as the time $t$ tends to infinity. 

\begin{proposition}\label{prop:limit-Laplace-transform}
  Let $(X_{t})_{t\geq}$ be an affine process associated with the admissible
  parameter set $(b,B,m,\mu)$ satisfying Assumption~\ref{assump:subcritical-pure-jump}.
  Then for all $u\in\cHplus$ and for all $x\in\cHplus$ we have
 \begin{align}\label{eq:limit-Laplace-transform-1}
  \lim_{t\to\infty}\EXspec{x}{\E^{-\langle u, X_{t}\rangle}}=\exp\left(
   -\int_{0}^{\infty}F(\psi(s,u))\,\D s\right)\in [0,\infty). 
 \end{align}
\end{proposition}
\begin{proof}
  Let $u\in\cHplus$ and $x\in\cHplus$, then by Lemma~\ref{lem:norm-bound-psi} and \eqref{eq:stable-semigroup} we have
\begin{align*}
  |\langle \psi(t,u),x\rangle|\leq \norm{\psi(t,u)}\norm{x}\leq
                                \norm{\E^{t\hat{B}}}_{\cL(\cH)}\norm{x}\norm{u}\leq M\E^{-\delta t} \norm{x}\norm{u} . 
\end{align*}
Lemma~\ref{lemma:RFcontinuous} gives
\begin{align}\label{eq:F integrability estimate}
 |F(\psi(t,u))|\leq C \left(\norm{\psi(t,u)}+\norm{\psi(t,u)}^{2}\right)\leq C M^{2}\E^{-\delta s}(\norm{u}+\norm{u}^{2}),
\end{align}
with $C=\norm{b}+ \int_{\cH^+\setminus \{0\}}\norm{\xi}^2\,m(\D\xi)$. For
every $u\in\cHplus$ this implies
\begin{align*}
\int_{0}^{\infty}|F(\psi(s,u))|\,\D s\leq \frac{C M^{2}}{\delta}(\norm{u}+\norm{u}^{2})<\infty,  
\end{align*}
and hence the limit
$\lim_{t\to\infty}\phi(t,u)=\int_{0}^{\infty}F(\psi(s,u))\,\D s$
exists for every $u\in\cHplus$. This, the continuity of the
exponential function and the fact that by Lemma~\ref{lem:norm-bound-psi}
$\langle \psi(t,u),x\rangle\to 0$ for all $x,u\in\cHplus$ as $t\to\infty$ imply~\eqref{eq:limit-Laplace-transform-1}.
\end{proof}

The next lemma asserts uniform boundedness in time of the transition semigroup
$(P_{t})_{t\geq 0}$ in the operator norm on $\Bsq$. 

\begin{lemma}\label{lem:uniform-growth-bound-semigroup}
  Let $(X_{t})_{t\geq}$ be an affine process associated with the admissible
  parameter set $(b,B,m,\mu)$ satisfying
  Assumption~\ref{assump:subcritical-pure-jump} and denote its transition
  semigroup by $(P_{t})_{t\geq 0}$. Then we have
  \begin{align}\label{eq:uniform-growth-bound-semigroup}
   \sup_{t\geq 0}\norm{P_{t}}_{\cL(\Bsq)}<\infty. 
  \end{align}
\end{lemma}
\begin{proof}
  Recall that $\rho(x)=1+\norm{x}^{2}$ and note that for every $f\in\Bsq$ we have $|f(y)| \leq \rho(y) \| f\|_{\mathcal{B}_{\rho}}$ and hence
  \begin{align*}
    \norm{P_{t}f}_{\Bsq} = \sup_{x\in\cHplus}\rho(x)^{-1}\left|\int_{\mathcal{H}^+}f(y)p_t(x,\D y)\right|
    \leq \norm{f}_{\Bsq}\norm{P_{t}\rho}_{\Bsq}, 
  \end{align*}
  which yields $\sup_{t\geq 0}\norm{P_{t}}_{\cL(\Bsq)}\leq \sup_{t\geq 0}\norm{P_{t}\rho}_{\Bsq}$.
  Let $(e_{i})_{i\in\MN}$ be an orthonormal basis of $\cH$ and recall that by
  \cite[Remark 4.6]{cox2020affine} we have $P_{t}\rho(x)=\EXspec{x}{\rho(X_{t})}$ for all $t\geq 0$.
  Hence by Parseval's identity we conclude 
  \begin{align*}
    0\leq
    P_{t}\rho(x)=1+\EXspec{x}{\norm{X_{t}}^{2}}=1+\sum_{i=1}^{\infty}\EXspec{x}{\langle
    X_{t},e_{i}\rangle^{2}}.  
  \end{align*}
  Using~\eqref{eq:explicit-Xv-square} with $v=w=e_{i}$ for $i\in\MN$ we find
  \begin{align*}
\EXspec{x}{\langle X_{t}, e_{i}\rangle^{2}}&= \Big( 
        \int_0^t \langle\hat{b},\E^{s \hat{B}} e_{i}\rangle\,\D s 
        + 
        \Big\langle x,
        \E^{t \hat{B}}e_{i}\Big\rangle\Big)^{2}+ \int_0^t\int_{\cHpluso}\langle \xi, \E^{s \hat{B}} e_{i}\rangle^{2}\,m(\D\xi)\D s
\nonumber \\ & \quad 
    +  \int_0^t \int_{0}^{s} 
        \Big\langle\hat{b},\E^{(s-u)\hat{B}}\int_{\cHpluso}\langle \xi,\E^{u \hat{B}}e_{i}\rangle^{2}\,\frac{\mu(\D\xi)}{\norm{\xi}^{2}}\Big\rangle\,\D u\, \D s
\nonumber\\ & \quad 
    + \int_{0}^{t}
        \Big\langle x,\E^{(t-s)\hat{B}}\int_{\cHpluso}\langle\xi,\E^{s
              \hat{B}}e_{i}\rangle^{2}\,\frac{\mu(\D\xi)}{\norm{\xi}^{2}}
              \Big\rangle\,\D s, 
 \end{align*}
 and hence
 \begin{align}
  \sum_{i=1}^{\infty}\EXspec{x}{\langle X_{t},e_{i}\rangle^{2}}&\leq
                                                                 2\norm{\int_{0}^{t}
                                                                 \E^{s\hat{B}^{*}}\hat{b}\,\D
                                                                 s}^{2}+2\norm{\E^{t\hat{B}^{*}}x}^{2}\label{eq:uniform-growth-bound-semigroup-1}\\
   &\quad +\int_{0}^{t}\int_{\cHpluso}\norm{\E^{s\hat{B}^{*}}\xi}^{2}\,m(\D\xi)\,\D
     s\label{eq:uniform-growth-bound-semigroup-2}\\
   &\quad +\int_{0}^{t}\int_{0}^{s}\int_{\cHpluso}\norm{\E^{u\hat{B}^{*}}\xi}^{2}\big\langle
     \hat{b},\E^{(s-u)\hat{B}}\frac{\mu(\D\xi)}{\norm{\xi}^{2}}\big\rangle\D
     u\D s\label{eq:uniform-growth-bound-semigroup-3}\\
   &\quad +\int_{0}^{t}\int_{\cHpluso}
     \norm{\E^{s\hat{B}^{*}}\xi}^{2}\big\langle x,
     \E^{(t-s)\hat{B}}\,\frac{\mu(\D\xi)}{\norm{\xi}^{2}}\big\rangle\,\D
     s.\label{eq:uniform-growth-bound-semigroup-4}
 \end{align}
 In the following we show that every term on the right-hand side
 of~\eqref{eq:uniform-growth-bound-semigroup-1}-\eqref{eq:uniform-growth-bound-semigroup-4}
 converges as $t\to\infty$ uniformly in $x$, which then
 yields~\eqref{eq:uniform-growth-bound-semigroup}. 
 
 Note first that the adjoint
 semigroup $(\E^{t\hat{B}^{*}})_{t\geq 0}$ generated by $\hat{B}^{*}$, the adjoint of $\hat{B}$, is also uniformly stable as $\norm{\E^{t\hat{B}}}_{\cL(\cH)}=\norm{\E^{t\hat{B}*}}_{\cL(\cH)}$ for all $t\geq 0$. For the first term on the right-hand side of~\eqref{eq:uniform-growth-bound-semigroup-1} we have
 $\int_{0}^{t}\left\| \E^{s\hat{B}^{*}}\hat{b}\right\| \,\D s \leq \frac{M}{\delta}\norm{\hat{b}}$. The second term in~\eqref{eq:uniform-growth-bound-semigroup-1} vanishes as $t\to\infty$, since $(\E^{t\hat{B}^{*}})_{t\geq 0}$ is uniformly stable.
 Note that $s\mapsto M^{2} \E^{-2\delta s}\int_{\cHpluso}\norm{\xi}^{2}\,m(\D\xi)$ is
 an integrable majorant for the term
 in~\eqref{eq:uniform-growth-bound-semigroup-2} and thus the integral converges for $t\to\infty$.
 For \eqref{eq:uniform-growth-bound-semigroup-3} note
 that $\langle
 \hat{b},\E^{(s-u)\hat{B}}\,\frac{\mu(\D\xi)}{\norm{\xi}^{2}}\big\rangle\geq
 0$ for every $s,u\in\MRplus$, which follows from the admissible parameter conditions, which imply that $\hat{b}\in\cHplus$ and $\E^{(s-u)\hat{B}}(\cHplus)\subseteq \cHplus$, whenever $s\geq u$. Hence we have 
 \begin{align*}
   &\int_{0}^{\infty}\int_{0}^{s}\int_{\cHpluso}\norm{\E^{u\hat{B}^{*}}\xi}^{2}\big\langle
     \hat{b},\E^{(s-u)\hat{B}}\,\frac{\mu(\D\xi)}{\norm{\xi}^{2}}\big\rangle\,\D
     u\,\D s\leq \frac{3}{2}\frac{M^{3}}{\delta^{2}}\norm{\hat{b}}\norm{\mu(\cHpluso)}.
     \end{align*}
  Finally note that $\int_{0}^{t}\E^{-2\delta s}\E^{-\delta (t-s)}\,\D
     s=\frac{1}{\delta}(\E^{-\delta t}-\E^{-2\delta t})$ and hence the last
     term in~\eqref{eq:uniform-growth-bound-semigroup-4} vanishes as
     $t\to\infty$, which can be seen from
     \begin{align}
      &\int_{0}^{t}\int_{\cHpluso}\norm{\E^{s\hat{B}^{*}}\xi}^{2}\big\langle
       x,
       \E^{(t-s)\hat{B}}\,\frac{\mu(\D\xi)}{\norm{\xi}^{2}}\big\rangle\,\D
        s\nonumber\\
       &\qquad\leq M^{3}\norm{\mu(\cHpluso)}\norm{x} \int_{0}^{t}\E^{-2\delta
          s}\E^{-\delta (t-s)}\,\D s\nonumber\\
       &\qquad\leq \frac{M^{3}}{\delta}\norm{\mu(\cHpluso)}\norm{x}
         (\E^{-\delta t}-\E^{-2\delta t}).\label{eq:uniform-growth-bound-semigroup-5} 
     \end{align}
     Thus we proved that $\sup_{t\geq
       0}\sup_{x\in\cHplus}\EXspec{x}{\rho(X_{t})}<\infty$, which proves the statement.
   \end{proof}

   In the next proposition we show first that for every $f\in\Bsq$ the transition
   semigroup $(P_{t})_{t\geq 0}$ converges in $\Bsq$ as $t\to\infty$ and
   subsequently use this to define a continuous linear functional on $\Bsq$
   given by the limits.

   \begin{proposition}\label{prop:convergence-semigroup}
     For all $f\in\Bsq$ the limit $\lim_{t\to\infty}P_{t}f$ exists in $\Bsq$ and
     $\pi(f)\df \lim_{t\to\infty}P_{t}f(x)$ defines a continuous linear
     functional on $\Bsq$. 
   \end{proposition}
   \begin{proof}
     By Proposition~\ref{prop:limit-Laplace-transform} we know that for every $u\in\cHplus$
     \begin{align*}
       \lim_{t\to\infty}(P_{t}\E^{-\langle
       u,\cdot\rangle})(x)=\E^{-\int_{0}^{\infty}F(\psi(s,u))\,\D s}, \quad
       \forall\, x\in\cHplus.
     \end{align*}
     Define, for $u\in\cHplus$, 
     $\pi_{u}=\E^{-\int_{0}^{\infty}F(\psi(s,u))\,\D s}\one$
     where $\one$ denotes the constant one function. We claim that the sequence $(P_{t}\E^{-\langle u, \cdot\rangle})_{t\geq 0}$ converges in $\Bsq$ to the constant function $\pi_u \in \mathcal{B}_{\rho}(\mathcal{H}_w^+)$. Indeed, we have
     \begin{align*}
       \norm{P_{t}\E^{-\langle
       u,\cdot\rangle}-\pi_{u}}_{\rho} &= \sup_{x\in\cHplus}\frac{\Big\lvert\big(\E^{-\int_{0}^{t}F(\psi(s,u))\,\D
                                            s-\langle
                                            \psi(t,u),x\rangle}-\E^{-\int_{0}^{\infty}F(\psi(s,u))\,\D
                                            s}\big)\Big\rvert}{\rho(x)}\nonumber\\
                                                                  &\leq
         \sup_{x\in\cHplus}\frac{\big\lvert\int_{t}^{\infty}F(\psi(s,u))\,\D s-\langle
         \psi(t,u),x\rangle\big\rvert}{\rho(x)}
         \\ &\leq \int_t^{\infty}|F(\psi(s,u))|\,\D  s + \| \psi(t,u)\|\sup_{x \in \cHplus} \frac{\|x\|}{\rho(x)}.
     \end{align*}
     where we have used $\rho(x) = 1 + \|x\|^2$.
     The first term converges to zero due to \eqref{eq:F integrability estimate}, while the second term tends to zero by Lemma~\ref{lem:norm-bound-psi}.
     
     Let $\cD\df\lin\set{\E^{-\langle u,\cdot \rangle}\colon u\in\cHplus}$ and define $\pi$ as the linear extension of $\pi_u$ onto $\cD$. In particular, we have
     $\lim_{t \to \infty}P_{t}f = \pi(f)$ in $\Bsq$ for every $f\in\cD$. In view of Proposition~\ref{lem:uniform-growth-bound-semigroup} we have $\sup_{t\geq 0}\norm{P_{t}}_{\cL(\Bsq)}<\infty$ and hence 
     $|\pi(f)| \leq \sup_{t\geq 0}\norm{P_{t}}_{\cL(\Bsq)}\norm{f}_{\Bsq}$, i.e.
     $\pi$ is bounded on $\cD$.
     Since $\cD$ is dense in $\Bsq$, see \cite[Lemma 4.7]{cox2020affine}, this means that there exists a
     unique extension of $\pi$ to a continuous linear functional on $\Bsq$, which we also denote by $\pi$. We thus proved the existence of $\pi\in
     \cL(\Bsq,\MR)$ and it is only left to show that $P_{t}f\to
     \pi(f)$ as $t\to\infty$ for all $f\in\Bsq$. 
     The latter one is an immediate consequence of an $\varepsilon/3$-argument using $\sup_{t\geq 0}\norm{P_{t}}_{\cL(\Bsq)}<\infty$ and $\overline{\cD} = \Bsq$. Thus we conclude the assertion.

   \end{proof}
   In the following lemma we prove that the functional $\pi$ is represented by a unique probability measure on $\cB(\cHplus)$:
   \begin{lemma}\label{lem:invariant-measure-functional}
   Let $\pi$ denote the continuous linear functional from
   Proposition~\ref{prop:convergence-semigroup}. Then there exists a unique probability measure $\nu$ on $\cHplus$ such that
   \begin{align}
     \label{eq:invariant-functional}
    \pi(f)=\int_{\cHplus}f(\xi)\,\nu(\D\xi)\quad \text{for all } f\in\Bsq, 
   \end{align}
   and $\nu$ is inner-regular on $\cB(\cHplus)$ when $\cHplus$ is equipped with the weak topology.      
   \end{lemma}
   \begin{proof}
     By an application of the Riesz-representation theorem in \cite[Theorem
     2.4]{CT20} there exists a unique finite signed Radon measure $\nu$ on
     $\cB(\cHplus)$ such that~\eqref{eq:invariant-functional} and 
     \begin{align}\label{eq:total-variation-norm}
       \int_{\cHplus}\big(1+\norm{x}^{2}\big)\,|\nu|(\D\xi)=\norm{\pi}_{\cL(\Bsq,\MR)} 
     \end{align}
     hold. Here $|\nu|$ denotes the total variation measure of $\nu$. Note that $\nu$ is
     a Radon measure with respect to the weak topology on $\cHplus$, which
     implies the statement on the inner-regularity. It is left to prove
     that $\nu$ is a probability measure. Note that since
     $\lim_{t\to\infty}P_{t}\one(x)=1$ we have $\pi(\one)=1$ and hence
     $\nu(\cHplus)=1$. Moreover, as $P_{t}f\geq 0$ for all non-negative $f\in
     C_{b}(\cHplus_{\mathrm{w}})$ and all $t\geq 0$, we have
     $\lim_{t\to\infty}P_{t}f(x)\geq 0$ for all $x\in\cHplus$ and hence
     $\int_{\cHplus}f(\xi)\,\nu(\D\xi)\geq 0$ for all non-negative $f\in
     C_{b}(\cHplus_{\mathrm{w}})$, which implies that the measure $\nu$ is also non-negative and hence it is a probability measure on $\cB(\cHplus)$. 
   \end{proof}
   
   In the following we identify the linear functional $\pi$ with the measure $\nu$ given by Lemma~\ref{lem:invariant-measure-functional} and write $\pi$
   instead of $\nu$. Finally we show that $\pi$ is, indeed, the unique invariant measure of $(p_{t}(x,\cdot))_{t\geq 0}$. 
   
   \begin{proposition}\label{prop:existence-invariant-measure}
     Let $(b,B,m,\mu)$ be an admissible parameter set such that
     Assumption~\ref{assump:subcritical-pure-jump} is satisfied and denote the
     associated subcritical affine Markov process on $\cHplus$ by $(X_{t})_{t\geq
       0}$ and its transition kernels by $(p_{t}(x,\D\xi))_{t\geq 0}$.   
     Then there exists a unique invariant measure $\pi$ for
     $(p_{t}(x,\cdot))_{t\geq 0}$. Moreover, for every $x\in\cHplus$ we have 
     \begin{align}\label{eq:convergence-to-invariant}
       \lim_{t\to\infty}\int_{\cHplus}f(\xi)\,p_{t}(x,\D\xi)\to \int_{\cHplus}f(\xi)\,\pi(\D\xi),\quad
       \forall f\in C_{b}(\cHplus_{\mathrm{w}}),  
     \end{align}
     and the Laplace transform of $\pi$ is given by~\eqref{eq:Laplace-invariant}.
   \end{proposition}
   \begin{proof}     
     In Proposition~\ref{prop:convergence-semigroup} and the subsequent
     arguments, we have already shown the existence of the Borel measure $\pi$ such
     that~\eqref{eq:convergence-to-invariant} holds. It is left to show that
     $\pi$ is the unique invariant measure. We have
     \begin{align*}
       \int_{\cHplus}\E^{-\langle
       u,\xi\rangle}\Big(\int_{\cHplus}p_{t}(x,\D\xi)\,\pi(\D
       x)\Big)&=\int_{\cHplus}\Big(\int_{\cHplus}\E^{-\langle
                u,\xi\rangle}\,p_{t}(x,\D\xi)\Big)\,\pi(\D x)\\
              &= \E^{-\phi(t,u)}\int_{\cHplus}\E^{-\langle
                x,\psi(t,u)\rangle}\,\pi(\D x).                                                        
     \end{align*}
     Note that by \eqref{eq:semi-flow} we have $\psi(t+s,u)=\psi(t,\psi(s,u))$
     and hence for every $u\in\cHplus$ we have
     \begin{align*}
       \E^{-\phi(t,u)}\int_{\cHplus}\E^{-\langle
                x,\psi(t,u)\rangle}\,\pi(\D x) &=\E^{-\phi(t,u)}\E^{
                                                       -\int_{0}^{\infty}F(\psi(s,\psi(t,u)))\,\D s
                                                       }\\
                                                     &=\E^{-\phi(t,u)}\E^{-\int_{0}^{\infty}F(\psi(t+s,u))\,\D
                                                       s}\\
                                                     &=\E^{-\phi(t,u)}\E^{-\int_{t}^{\infty}F(\psi(s,u))\,\D
                                                       s}\\
                                                     &= \E^{-\int_{0}^{\infty}F(\psi(s,u))\,\D
                                                       s}\\
                                                     &=
                                                       \int_{\cHplus}\E^{-\langle
                                                       x,u\rangle}\,\pi(\D x).
  \end{align*}
  This proves the invariance of $\pi$. Next, we prove that $\pi$ is the unique
  invariant measure. Suppose there exists a $\pi'\in\cM(\cHplus)$ which is
  invariant for $p_{t}(x,\D\xi)$, then for every $u\in\cHplus$ and $t\geq 0$
  we have:
  \begin{align*}
    \int_{\cHplus}\E^{-\langle x, u\rangle}\,\pi'(\D
    x)&=\int_{\cHpluso}\E^{-\langle
    u,\xi\rangle}\Big(\int_{\cHpluso}p_{t}(x,\D\xi)\,\pi'(\D x)\Big)\\
    &=\int_{\cHpluso}\E^{-\phi(t,u)-\langle x,\psi(t,u)\rangle}\,\pi'(\D x),
  \end{align*}
  now by letting $t\to\infty$ we find that
  \begin{align*}
  \int_{\cHpluso}\E^{-\langle u,x\rangle}\,\pi'(\D x)=\exp\left(-\int_{0}^{\infty}F(\psi(s,u))\D s\right)\,.  
  \end{align*}
  The Laplace transform is measure determining for measures on $\cB(\cHplus)$ and hence $\pi=\pi'$.
\end{proof}

\begin{remark}\label{rem:weak-convergence}
The convergence in~\eqref{eq:convergence-to-invariant}
is weak convergence of $p_{t}(x,\cdot)$ to $\pi$ as $t\to\infty$ in the weak topology
on $\cH$. Even though the Borel algebras of $\cH$ equipped with the norm
topology and weak topology coincide, the weak convergence is
different in general. We say that $p_{t}(x,\cdot)\to \pi$ as $t\to\infty$
\textit{weakly in the weak topology} on $\cHplus$, whenever $P_{t}f(x)\to
\int_{\cHplus}f(\xi)\,\pi(\D \xi)$ for all $f\in C_{b}(\cHplus_{\mathrm{w}})$. 

If the stronger assumption $P_{t}f(x)\to \int_{\cHplus}f(\xi)\,\pi(\D \xi)$ for all
$f\in C_{b}(\cHplus)$ holds, we speak of the usual \textit{weak
  convergence}, i.e., $p_{t}(x,\cdot)\Rightarrow \pi$ as $t\to\infty$. By
\cite[Theorem 1 and  2]{Mer89} we know that weak convergence in the weak
topology together with 
\begin{align*}
  \lim_{N\to\infty}\sup_{n\in\MN}p_{t}(x, A_{N,n})=0,\quad \text{for all }\epsilon>0, 
\end{align*}
where $A_{N,n}\df\set{\sum_{i=N}^{\infty}\langle x, e_{i}\rangle^{2}\geq
  \epsilon}$ for $N,n\in\MN$, implies $p_{t}(x,\cdot)\Rightarrow \pi$ as
$t\to\infty$. Note that in our main Theorem~\ref{thm:convergence-Wasserstein}
we assert weak convergence in the \emph{strong topology}, which will be shown below.
\end{remark}

\subsection{Proof of Theorem~\ref{thm:convergence-Wasserstein}}\label{sec:conv-rates-wass}
Proposition~\ref{prop:existence-invariant-measure} ensures the existence of a unique invariant measure $\pi$ of
$(p_{t}(x,\cdot))_{t\geq 0}$ with Laplace transform~\eqref{eq:Laplace-invariant}. We also proved weak convergence of $(p_{t}(x,\cdot))_{t\geq 0}$
to $\pi$ as $t\to\infty$ in the weak topology. What is left to show is the convergence rates in Wasserstein
distance of order $p$ for $p\in [1,2]$ as
in~\eqref{eq:convergence-Wasserstein}. Then convergence in Wasserstein distance of some order $p\in [1,\infty)$ implies weak convergence (in the
strong topology) and convergence of the $p$-th absolute moment, see \cite[Theorem
6.9]{Vil09}. This implies the last assertion of
Theorem~\ref{thm:convergence-Wasserstein}. In the remainder we prove the
convergence rates~\eqref{eq:convergence-Wasserstein}.

Let $p\in [1,2]$ and as before we denote by $q_{t}(x,\D\xi)$ the transition
kernel of an affine process associated with the admissible parameter set
$(0,B,0,\mu)$. Let $t\geq 0$, $x\in\cHplus$ and $G\in\cC(\delta_{x},\pi)$ i.e. $G$ is a coupling with
marginals $\delta_{x}$ and $\pi$. Note that
\begin{align*}
p_{t}(x,\D y)=\int_{\cHplus}p_{t}(z,\D y)\,\delta_{x}(\D z)=\int_{\cHplus\times\cHplus}p_{t}(z,\D y)\one(z') G(\D z, \D z')  
\end{align*}
and by the invariance of $\pi$ we also have
\begin{align*}
  \pi(\D y)=\int_{\cHplus}p_{t}(z',\D y)\,\pi(\D z')=\int_{\cHplus\times\cHplus}p_{t}(z',\D y)\one(z) G(\D z, \D z').
\end{align*}
  Thus by the convexity property in \cite[Theorem 4.8]{Vil09} and since $W_{p}\leq W_{2}$ for $p\in [1,2]$ we have 
  \begin{align}
    W_{p}\big(p_{t}(x,\cdot),\pi\big)&=W_{p}\left(\int_{\cHplus}p_{t}(z,\cdot)\,\delta_{x}(\D
                                       z),\int_{\cHplus}p_{t}(y,\cdot)\,\pi(\D
                                       y)\right)\nonumber\\
                                     &\leq
                                       \left(\int_{\cHplus\times\cHplus}W_{2}\big(p_{t}(z,\cdot),p_{t}(y,\cdot)\big)^{p}\,G(\D
                                       z, \D y) \right)^{1/p}.\label{eq:convergence-Wasserstein-1}
  \end{align}
  By Lemma~\ref{lem:convolution-homogeneous} we have
  $p_{t}(z,\cdot)=q_{t}(z,\cdot)*p_{t}(0,\cdot)$ for every $t\geq 0$. Thus for
  $H\in\cC(q_{t}(z,\cdot),q_{t}(y,\cdot))$ we obtain by Lemma~\ref{lem:Wasserstein-convolution} that
  \begin{align}
    W_{2}\big(p_{t}(z,\cdot),p_{t}(y,\cdot)\big)^{p}
    &= W_{2}\big(q_{t}(z,\cdot)*p_{t}(0,\cdot),q_{t}(y,\cdot)*p_{t}(0,\cdot)\big)^{p}\nonumber
    \\ &\leq W_{2}\big(q_{t}(z,\cdot),q_{t}(y,\cdot)\big)^{p}\nonumber
    \\ &\leq \left(\int_{\cHplus\times\cHplus}\norm{\tilde{x}-\tilde{y}}^2\,H(\D \tilde{x}, \D \tilde{y})\right)^{p/2}\nonumber
    \\ &\leq \left(2\int_{\cHplus\times\cHplus}\big(\norm{\tilde{x}}^2+\norm{\tilde{y}}^2\big)\,H(\D \tilde{x}, \D \tilde{y})\right)^{p/2}\nonumber
    \\ &= \left(2\int_{\cHplus\times\cHplus}\norm{\tilde{x}}^2\,q_{t}(z, \D                                 \tilde{x})+2\int_{\cHplus\times\cHplus}\norm{\tilde{y}}^2\,q_{t}(y,\D \tilde{y})\right)^{p/2}.\label{eq:convergence-Wasserstein-2-1}
  \end{align}
  Now, recall from~\eqref{eq:uniform-growth-bound-semigroup-1} that
  \begin{align*}
   \int_{\cHplus\times\cHplus}\norm{\tilde{x}}^2\,q_{t}(z,
    \D\tilde{x}) \leq  2\norm{\E^{t\hat{B}^{*}}z}^{2} + \int_{0}^{t}\int_{\cHpluso}\norm{\E^{s\hat{B}^{*}}\xi}^{2}\big\langle
    z,\E^{(t-s)\hat{B}}\,\frac{\mu(\D\xi)}{\norm{\xi}^{2}}\big\rangle\,\D s, 
  \end{align*}
  while all the other terms vanish
  as $\hat{b}=0$ and $m = 0$. By the same estimations as
  in~\eqref{eq:uniform-growth-bound-semigroup-5} we conclude that
  \begin{align*}
  \int_{\cHplus\times\cHplus}\norm{\tilde{x}}^2\,q_{t}(z, \D
    \tilde{x})&\leq   2\norm{\E^{t\hat{B}^{*}}z}^{2}+\int_{0}^{t}\int_{\cHpluso}\norm{\E^{s\hat{B}^{*}}\xi}^{2}\big\langle
    z,\E^{(t-s)\hat{B}}\,\frac{\mu(\D\xi)}{\norm{\xi}^{2}}\big\rangle\,\D
    s\\
    &\leq 2 M^2\E^{-2\delta
      t}\norm{z}^{2}+\frac{M^{3}}{\delta}\norm{\mu(\cHpluso)}\E^{-\delta t}\norm{z}.
  \end{align*}
  Inserting this back
  into~\eqref{eq:convergence-Wasserstein-2-1} and using the sub-additivity of $x\mapsto
  x^{p/2}$ for $p\in [1,2]$, we obtain
  \begin{align}
    W_{2}\big(p_{t}(z,\cdot),p_{t}(y,\cdot)\big)^{p}&\leq \big(C_{1}\E^{-\delta
    t}\norm{z}\big)^{p}+\big(C_{2}\E^{-\delta/2}\norm{z}^{1/2}\big)^{p}\nonumber\\
    &\quad+\big(C_{1}\E^{-\delta
    t}\norm{y}\big)^{p}+\big(C_{2}\E^{-\delta/2}\norm{y}^{1/2}\big)^{p}, \label{eq:convergence-Wasserstein-3}
  \end{align}
  for $C_{1}=2M$ and $C_{2}=2^{1/2}M^{3/2}\delta^{-1/2}\norm{\mu(\cHpluso)}^{1/2}.$
  Now, plugging~\eqref{eq:convergence-Wasserstein-3} back
  into~\eqref{eq:convergence-Wasserstein-1} and again by the sub-additivity of $x\mapsto
  x^{1/p}$, we obtain the desired~\eqref{eq:convergence-Wasserstein}.

\subsection{Proof of Corollary~\ref{cor:stationary}}\label{sec:proof-coroll-refc}

For every $x\in\cHplus$ let $(p_{t}(x,\cdot))_{t\geq 0}$ be the transition
kernels associated to the admissible parameter set $(b,B,m,\mu)$ by
Theorem~\ref{thm:CKK20}. Moreover, let $\pi$ be the unique invariant distribution of
$(p_{t}(x,\cdot))_{t\geq 0}$ (which is independent of $x\in\cHplus$). From~\cref{item:invariant-measure} 
we know that $\pi$ is inner-regular. Thus from
Proposition~\ref{prop:initial-distribution} we conclude the existence of a
unique Markov process $(X^{\pi})_{t\geq 0}$ such that for all $f\in\Bsq$ we have
$\EXspec{\pi}{f(X_{t})}=\int_{\cHplus}P_{t}f(x)\,\pi(\D x)$.
Moreover, since $\pi$ is the invariant measure we have for each $t \geq 0$
\begin{align*}
\int_{\cHplus}P_{t}f(x)\,\pi(\D x)=\int_{\cHplus}\Big(\int_{\cHplus}f(\xi)
  p_{t}(x,\xi)\Big)\,\pi(\D x) =\int_{\cHplus}f(\xi)\,\pi(\D \xi),
\end{align*}
which implies that  for all $t\geq 0$ the random variable $X_{t}^{\pi}$ has
distribution $\pi$.

\subsection{Proof of Proposition~\ref{prop:limit-moment-formula}}\label{sec:proof-prop-refpr}
Let us denote the space of all Hilbert-Schmidt operators on $\cH$ by
  $\cL_{2}(\cH)$ and note that $(e_{i}\otimes e_{j})_{i,j\in\MN}$ is an
  orthonormal basis of $\cL_{2}(\cH)$. For every $y\in\cH$ the operator
  $y\otimes y\colon\cH\to\cH$ defined by $y\otimes y (x)=\langle x,y\rangle y$
  for every $x\in\cHplus$ is a Hilbert-Schmidt operator on $\cH$ and we can write
  $y\otimes y=
  \sum_{i,j=1}^{\infty}\langle y, e_{i} \rangle\langle y, e_{j}\rangle
  e_{i}\otimes e_{j}$. Note that by~\eqref{eq:total-variation-norm} we have
  $\int_{\cHplus}\rho(\xi)\,\pi(\D\xi)=\norm{\pi}_{\cL(\Bsq,\MR)}<\infty$ and
  hence the absolute second moment of $\pi$ is finite, which implies
  \begin{align*}
   \int_{\cHplus}\norm{y\otimes y}_{\cL_{2}(\cH)}\,\pi(\D y)\leq \int_{\cHplus}\Tr(y\otimes y)\,\pi(\D
    y)\leq \int_{\cHplus}\norm{y}^{2}\,\pi(\D y)<\infty,    
  \end{align*}
  and hence the integral $\int_{\cHplus}y\otimes y\,\pi(\D y)$ is well-defined in the
  Bochner sense. Thus it remains to compute the first two moments of the invariant distribution.
  
  Note that for every $u\in\cH$ the linear functional $\langle
  u,\cdot\rangle\colon\cH\to\MR$ satisfies the following two properties:
  \begin{enumerate}
  \item[i)] for every $R>0$ we have $\langle u,\cdot\rangle\in C_{b}(K_{\mathrm{w}}^{R})$ where the set
    \begin{align*}
      K^{R}_{\mathrm{w}}\df\set{x\in\cHplus\colon\,\norm{x}^{2}+1\leq R}  
    \end{align*}
    is compact in $\cHplus$ equipped with the weak topology and
  \item[ii)] $\lim_{R\to\infty}\sup_{x\in\cHplus\setminus K^{R}_{\mathrm{w}}}|\langle u,x\rangle|(1+\norm{x}^{2})^{-1}=0.$   
  \end{enumerate}
  which by \cite[Theorem 2.7]{doersek2010semigroup} implies $\langle u,\cdot\rangle\in\Bsq$
  for all $u\in\cH$. By Proposition~\ref{prop:convergence-semigroup} we have
  $P_{t}f\to \pi(f)$ as $t\to\infty$ for all $f\in\Bsq$ and hence also
  $P_{t}\langle u, \cdot\rangle\to \int_{\cHplus}\langle u,\xi\rangle\,\pi(\D\xi)$
  as $t\to\infty$. Let $(e_{i})_{i\in\MN}$ be an orthonormal basis of $\cH$. Then by~\eqref{eq:explicit-Xv} for $u=e_{i}$ for $i\in\MN$ we have
  \begin{align*}
  \lim_{t\to\infty}P_{t}\langle e_{i},
    \cdot\rangle&=\lim_{t\to\infty}\big(\int_{0}^{t}\langle \hat{b},\E^{s\hat{B}}e_{i}\rangle\,\D s+\langle x,
              \E^{t\hat{B}}e_{i}\rangle\big)=\int_{0}^{\infty}\langle \hat{b},\E^{s\hat{B}}e_{i}\rangle\,\D s     
  \end{align*}
  and since $\xi=\sum_{i=1}^{\infty}\langle \xi, e_{i}\rangle e_{i}$ it
  follows that
  \begin{align*}
  \lim_{t\to\infty}\int_{\cHplus}\xi\,p_{t}(x,\D\xi)
  &= \lim_{t \to \infty}\sum_{i=1}^{\infty}\int_{\cHplus}\langle
                                     \xi, e_{i}\rangle e_{i}\,p_{t}(x,\D\xi)
                                   \\ &=\sum_{i=1}^{\infty
                                     }\Big(\int_{0}^{\infty}\langle
                                     \hat{b},\E^{s\hat{B}}e_{i}\rangle\,\D
                                     s\Big) e_{i}
                                     \\ &=\int_{0}^{\infty}\E^{s\hat{B}^{*}}\hat{b}\,\D s,
  \end{align*}
  where we have used 
  $\lim_{t \to \infty}\sum_{i=1}^{\infty}\Big(\int_{0}^{t}\langle \hat{b},\E^{s\hat{B}}e_{i}\rangle\,\D s\Big) e_{i}
  = \sum_{i=1}^{\infty}\Big(\int_{0}^{\infty}\langle \hat{b},\E^{s\hat{B}}e_{i}\rangle\,\D s\Big) e_{i}$ which is justified if $\lim_{N \to \infty}\sup_{t \geq 0} \sum_{i=N}^{\infty}\left\| \int_0^t \langle e^{s \widehat{B}^*}\widehat{b}, e_i \rangle e_i \D s \right\| = 0$. The latter one follows from
  $\sup_{t \geq 0} \sum_{i=N}^{\infty}\left\| \int_0^t \langle e^{s \widehat{B}^*}\widehat{b}, e_i \rangle e_i \,\D s \right\| \leq \int_0^{\infty} \sum_{i=1}^N | \langle e^{s \widehat{B}^*}\widehat{b}, e_i \rangle | \,\D s$ and
  \[
   \int_0^{\infty} \sum_{i=1}^{\infty} | \langle e^{s \widehat{B}^*}\widehat{b}, e_i \rangle | \D s
   \leq \int_0^{\infty} \| e^{s \widehat{B}^*}\widehat{b} \|\, \D s
   \leq M \| \widehat{b}\| \delta^{-1} < \infty.
  \]
  Recalling that $\hat{b}=b+\int_{\cHplus\cap\{\norm{\xi}>1\}}\xi\,m(\D\xi)$
  yields~\eqref{eq:limit-first-moment-formula}. 
  
  Next we prove the desired formula for the second moments of $\pi$.
  For $i,j\in\MN$ we set $g^{i,j}\df\langle\cdot, e_{i}\rangle\langle \cdot,e_{j}\rangle$. 
  From ~\eqref{eq:explicit-Xv-square} and analogous arguments as we used in Lemma~\ref{lem:uniform-growth-bound-semigroup} (to show that the integrals on the right-hand side of~\eqref{eq:limit-formula-2} below exists and are finite), we find that
  \begin{align}\label{eq:limit-formula-2}
    \lim_{t\to\infty}P_{t}g^{i,j}(x)&=\Big( 
                                                                 \int_0^\infty \langle\hat{b},\E^{s \hat{B}} e_{i}\rangle\,\D s 
                                                                 \Big)\Big( 
                                                                 \int_0^\infty \langle\hat{b},\E^{s \hat{B}} e_{j}\rangle\,\D s 
                                                                 \Big)\nonumber\\
   & \quad+ \int_0^\infty\int_{\cHpluso}\langle \xi, \E^{s \hat{B}} e_{i}\rangle\langle \xi, \E^{s \hat{B}} e_{j}\rangle\,m(\D\xi)\,\D s
\nonumber \\ & \quad 
    +  \int_0^\infty \int_{0}^{s} 
        \Big\langle\hat{b},\E^{(s-u)\hat{B}}\int_{\cHpluso}\langle \xi,\E^{u \hat{B}}e_{i}\rangle\langle \xi,\E^{u \hat{B}}e_{j}\rangle\,\frac{\mu(\D\xi)}{\norm{\xi}^{2}}\Big\rangle\,\D u\, \D s.
  \end{align}
  holds for all $i,j\in\MN$. The second moment formula~\eqref{eq:limit-second-moment-formula} then follows from this and $y\otimes y= \sum_{i,j=1}^{\infty}\langle y, e_{i} \rangle\langle y, e_{j}\rangle e_{i}\otimes e_{j}$, once we have shown that
  \begin{align}\label{eq auxiliary}
   \lim_{t\to\infty}P_{t}g^{i,j}(x) = \int_{\cHplus} g^{i,j}(y)\,\pi(\D y), \qquad i,j \in \MN.
  \end{align}
  Since the function $g^{i,j}$ does not belong to $\Bsq$, we cannot obtain \eqref{eq auxiliary}
  directly from Proposition~\ref{prop:convergence-semigroup}. However,
  since we have $P_{t}\langle \cdot,e_{i}\rangle\langle \cdot,e_{j}\rangle(x) \leq
  P_{t}\rho(x)<\infty$ for all $t\geq 0$ and $x\in\cHplus$, we see that the function is in the larger space $B_{\rho}(\cHplus_\mathrm{w})$ and we deduce the assertion by an additional approximation argument. 
  Namely, define $g_{n}^{i,j} \df g^{i,j} \wedge n$ for $n\in\MN$. Then $g_{n}^{i,j}\in\Bsq$ and we find that  
  \begin{align*}
      \left|P_t g^{i,j}(x) - \int_{\cHplus}g^{i,j}(x)\,\pi(\D x)\right|
      &\leq |P_t g^{i,j}(x) - P_t g_n^{i,j}(x)| 
      \\ &\ \ \ + \left|P_t g_n^{i,j}(x) - \int_{\cHplus}g^{i,j}_n(x)\,\pi(\D x)\right| 
      \\ &\ \ \ + \left| \int_{\cHplus}g^{i,j}_n(x)\,\pi(\D x) - \int_{\cHplus}g^{i,j}(x)\,\pi(\D x) \right|.
  \end{align*}
  Let $\varepsilon > 0$. Take $n \in \N$ large enough so that $\left| \int_{\cHplus}(g^{i,j}_n(x) - g^{i,j}(x))\,\pi(\D x) \right| < \varepsilon$.
  Next, note that 
  \[
   \lim_{n \to \infty}\sup_{t \geq 0}|P_t g^{i,j}(x) - P_t g_n^{i,j}(x)| \leq \lim_{n \to \infty}\sup_{t \geq 0}\mathbb{E}\left[ \| X_t\|^2 \mathbbm{1}_{\{ \| X_t\|^2 > n\}} \right] = 0
  \]
  where the last identity follows from the characterization of convergence in the Wasserstein distance (see \cite[Section 6]{Vil09}). Hence we find $n$ large enough such that 
  $|P_t g^{i,j}(x) - P_t g_n^{i,j}(x)| < \varepsilon$ holds uniformly in $t \geq 0$. 
  Finally, for this fixed choice of $n$, we may choose in view of Proposition~\ref{prop:convergence-semigroup} $t$ large enough so that
  $\left|P_t g^{i,j}(x) - \int_{\cHplus}g^{i,j}(x)\,\pi(\D x)\right| < \varepsilon$. 
  Combining all these estimates proves \eqref{eq auxiliary}.
  
 This completes the proof of Proposition \ref{prop:limit-moment-formula}.

\appendix

\section{Generalized Feller semigroups}\label{sec:gener-fell-semigr}
Let $(Y,\tau)$ be a completely regular Hausdorff topological space. We define:
 \begin{definition}
A function $\rho\colon Y \to (0,\infty)$ such that for every $R>0$ the set
$K_{R}\df\set{x\in Y:\;\rho(x)\leq R}$ is compact is called an \emph{admissible
weight function}. The pair $(Y,\rho)$ is called \emph{weighted space}. 
\end{definition}

Let $\rho\colon Y \rightarrow (0,\infty)$ be an admissible weight function and
assume that $\rho(x)\geq 1$ for every $x\in\cHplus$.
For $f\colon Y \rightarrow \MR$ we define $\| f \|_{\rho} \in [0,\infty]$ by
\begin{align}\label{eq:weighted_norm}
 \| f \|_{\rho} \df \sup_{x\in Y} \tfrac{|f(x)|}{\rho(x)}.
\end{align}
Note that $\| \cdot \|_{\rho}$ defines a norm on the vector space 
$
  B_{\rho}(Y)\df
  \set{
    f\colon Y\to \MR 
    \colon 
    \| f \|_{\rho} < \infty
    }
$
which renders $(B_{\rho}(Y),\|\cdot \|_{\rho})$ a Banach space.
Recall that $C_b(Y)$ denotes the space of bounded $\R$-valued $\tau$-continuous functions on $Y$. 
As the admissible weight function $\rho$ satisfies $\inf_{x\in Y}\rho(x)>0$, we have that $C_b(Y) \subseteq  B_{\rho}(Y)$.
\begin{definition}
 We define $\cB_{\rho}(Y)$ to be the closure of $C_{b}(Y)$ in $B_{\rho}(Y)$.
 \end{definition}
The following useful characterization of $\cB_{\rho}(Y)$ is proven
in~\cite[Theorem 2.7]{doersek2010semigroup}:
\begin{theorem}\label{thm:charac_Brho}
 Let $(Y,\rho)$ be a weighted space. Then $f\in \cB_{\rho}(Y)$
 if and only if $f\rvert_{K_{R}}\in C(K_{R})$ for all $R>0$
and 
\begin{align}\label{eq:function-growth}
\lim_{R\to\infty}\sup_{x\in Y\setminus K_{R}}\tfrac{|f(x)|}{\rho(x)}=0\,.
\end{align}
\end{theorem}
Next we recall the definition of a generalized Feller semigroup, as introduced in \cite[Section 3]{doersek2010semigroup}.
\begin{definition}\label{def:genFeller-semigroup}
A family of bounded linear operators $(P_t)_{t\geq 0}$ in $\cL(\cB_{\rho}(Y))$
is called a {\it generalized Feller semigroup (on $\cB_{\rho}(Y)$)}, if
\begin{enumerate}
\item \label{item:semigroup1} $P_0 = I$, the identity on $\cB_{\rho}(Y)$, 
\item \label{item:semigroup2}  $P_{t+s} = P_t P_s$ for all $t,s \geq 0$, 
\item \label{item:semigroup3}  $\lim\limits_{t\to 0+}P_{t}f(x)=f(x)$ for all $f\in \cB_{\rho}(Y)$ and $x\in Y$,
\item \label{item:semigroup4} there exist constants $C\in\MR$ and $\varepsilon>0$ such that $\norm{P_{t}}_{\cL(\cB_{\rho}(Y))}\leq C$ for all $t\in [0,\varepsilon]$,
\item \label{item:semigroup5}  $(P_{t})_{t\geq 0}$ is a positive semigroup,
  i.e., $P_{t}f\geq 0$ for all $t\geq 0$ and for all $f\in \cB_{\rho}(Y)$ satisfying $f\geq0$.
\end{enumerate}
\end{definition}

We make use of the following adapted version of the Kolmogorov extension
theorem presented in \cite[Theorem 2.11]{CT20}.

\begin{proposition}\label{prop:initial-distribution}
  Let $(Y,\rho)$ be a weighted space and let $(P_{t})_{t\geq 0}$ be a
  generalized Feller semigroup on $\cB_{\rho}(Y)$ with $P_{t}\one=\one$ for
  $t\geq 0$. Then for every $\nu\in\cM(Y)$ which is inner-regular, there
  exists a probability space $(\Omega,\cF,\MP_{\nu})$, filtered by a
  right-continuous filtration $(\cF_{t})_{t\geq 0}$, and a Markov process
  $(X_{t})_{t\geq 0}$ with values in $Y$ such that $\MP_{\nu}(X_{0}\in
  A)=\nu(A)$ for every $\cB(Y)$ and 
  \begin{align*}
    \EXspec{\MP_{\nu}}{f(X_{t})}=\int_{Y}P_{t}f(\xi)\,\nu(\D\xi),\quad
    t\geq 0,\, f\in\cB_{\rho}(Y).
  \end{align*}
\end{proposition}
\begin{proof}
In \cite[Theorem 2.11]{CT20} this version of the Kolmogorov
extension theorem is proven for $\nu=\delta_{x}$ and $x\in Y$. In the proof of
\cite[Theorem 2.11]{CT20} it is shown that the transition kernels
$p_{t}(x,\cdot)$ for $x\in Y$ and $t\geq 0$ given through the relation
\begin{align*}
P_{t}f(x)=\int_{Y}f(\xi)\,p_{t}(x,\D \xi),  
\end{align*}
form a Kolmogorov consistent family according to \cite[Section 15.6]{AB06}.
It is then concluded from a Kolmogorov extension theorem given by \cite[Theorem
15.23]{AB06} that there exists a probability measure $\MP_{\delta_{x}}$ such
that the assertions of the proposition hold. We draw the same conclusion for any other probability measures $\nu$ in
$\cM(Y)$ satisfying  
\begin{align}\label{eq:comapct-class-tight}
\nu(A)=\sup\set{\mu(K)\colon K\subseteq A\colon\, A\in\mathcal{K}\text{ and }A\in \cB(Y)\cap
  \mathcal{K}},  
\end{align}

where $\mathcal{K}$ is a compact class in $Y$. Note that every weighted
space $Y$ is a Hausdorff topological space and hence the family $\cK$ of all
compact sets of $Y$ forms a compact class, see \cite[Theorem 2.31]{AB06}. We
thus see that every inner-regular probability measure $\nu\in\cM(Y)$
satisfies~\eqref{eq:comapct-class-tight} and hence the assertion of the
proposition follows from this, Kolmogorov extension theorem \cite[Theorem
15.23]{AB06} and analogous arguments as in \cite[Theorem 2.11]{CT20}.
\end{proof}  

\section{A property of the Wasserstein distance}\label{sec:some-prop-wass-1}

\begin{lemma}\label{lem:Wasserstein-convolution}
  Let $W_2$ be the Wasserstein distance on $\mathcal{H}^+$. Let $\mu,\nu,\rho$ be Borel probability measures on $\mathcal{H}^+$. Then $W_2( \rho \ast \mu, \rho\ast \nu) \leq W_2(\mu,\nu).$
\end{lemma}
\begin{proof}
 Let $G$ be any coupling of $(\mu,\nu)$ and let $G'$ be any coupling of $(\rho,\rho)$. For each $f,g: \mathcal{H} \longrightarrow \R_+$ we find that
 \begin{align*}
   &\   \int_{\mathcal{H}^+ \times \mathcal{H}^+} \left( f(x) + g(y) \right)(G'\ast G)(\D x,\D y)
     \\ &= \int_{\mathcal{H}^+ \times \mathcal{H}^+}\int_{\mathcal{H}^+ \times \mathcal{H}^+} \left( f(x+z) + g(y+z') \right)G'(\D z,\D z')G(\D x,\D y)
     \\ &= \int_{\mathcal{H}^+ \times \mathcal{H}^+}f(x+z)\rho(\D z)\mu(\D x) + \int_{\mathcal{H}^+ \times \mathcal{H}^+}g(y+z')\rho(\D z')\nu(\D y)
 \end{align*}
 which shows that $G' \ast G$ is a coupling of $(\rho\ast \mu, \rho\ast \nu)$. Hence
 \begin{align*}
   &\  W_2(\rho \ast \mu, \rho \ast \nu)^2
    \\ &\leq \int_{\mathcal{H}^+ \times \mathcal{H}^+} \| x-y\|^2 (G' \ast G)(\D x,\D y)
     \\ &= \int_{\mathcal{H}^+ \times \mathcal{H}^+}\int_{\mathcal{H}^+ \times \mathcal{H}^+} \| (x+z) - (y+z')\|^2 G'(\D z,\D z')G(\D x,\D y)
     \\ &= \int_{\mathcal{H}^+ \times \mathcal{H}^+}\int_{\mathcal{H}^+ \times \mathcal{H}^+} \left( \| x-y\|^2 + 2\langle x-y, z-z'\rangle + \| z-z'\|^2 \right) G'(\D z,\D z')G(\D x,\D y)
     \\ &= \int_{\mathcal{H}^+ \times \mathcal{H}^+} \| x-y\|^2 G(\D ,\D y) + \int_{\mathcal{H}^+ \times \mathcal{H}^+} \|z-z'\|^2 G'(\D z,\D z') 
 \end{align*}
 where the last inequality follows from the fact that $G'$ has the same marginals so that
 $\int_{\mathcal{H}^+ \times \mathcal{H}^+} \langle x-y,z-z'\rangle G'(\D z,\D z') = 0$.
 Letting now $G'$ be the specific coupling determined by $G'(A\times B) = \rho\left( \{ z \in \mathcal{H}^+ \ : \ z \in A \cap B \}\right)$ with $A,B \in \mathcal{B}(\mathcal{H}^+)$,
 shows that $\int_{\mathcal{H}^+ \times \mathcal{H}^+} \|z-z'\|^2 G'(\D z,\D z') =
 0$. Since $G$ was arbitrary, the assertion is proved. 
\end{proof}

\bibliographystyle{acm}
\bibliography{LongTimeBehavior}

\begin{thebibliography}{10}

\bibitem{AKR16}
{\sc {Alfonsi}, A., {Kebaier}, A., and {Rey}, C.}
\newblock {Maximum likelihood estimation for Wishart processes}.
\newblock {\em {Stochastic Processes Appl.} 126}, 11 (2016), 3243--3282.

\bibitem{AB06}
{\sc {Aliprantis}, C.~D., and {Border}, K.~C.}
\newblock {\em {Infinite dimensional analysis. A hitchhiker's guide.}}
\newblock Berlin: Springer, 2006.

\bibitem{BDLP14}
{\sc {Barczy}, M., {D\"oring}, L., {Li}, Z., and {Pap}, G.}
\newblock {Stationarity and ergodicity for an affine two-factor model}.
\newblock {\em {Adv. Appl. Probab.} 46}, 3 (2014), 878--898.

\bibitem{BNS07}
{\sc {Barndorff-Nielsen}, O.~E., and {Stelzer}, R.}
\newblock {Positive-definite matrix processes of finite variation}.
\newblock {\em {Probab. Math. Stat.} 27}, 1 (2007), 3--43.

\bibitem{BDS55}
{\sc {Bartle}, R.~G., {Dunford}, N., and {Schwartz}, J.}
\newblock {Weak compactness and vector measures}.
\newblock {\em {Can. J. Math.} 7\/} (1955), 289--305.

\bibitem{BK14}
{\sc {Benth}, F.~E., and {Kr\"uhner}, P.}
\newblock {Representation of infinite-dimensional forward price models in
  commodity markets}.
\newblock {\em {Commun. Math. Stat.} 2}, 1 (2014), 47--106.

\bibitem{BRS18}
{\sc {Benth}, F.~E., {R\"udiger}, B., and {S\"uss}, A.}
\newblock {Ornstein-Uhlenbeck processes in Hilbert space with non-Gaussian
  stochastic volatility}.
\newblock {\em {Stochastic Processes Appl.} 128}, 2 (2018), 461--486.

\bibitem{benth2021barndorff}
{\sc Benth, F.~E., and Sgarra, C.}
\newblock A {B}arndorff-{N}ielsen and {S}hephard model with leverage in
  {H}ilbert space for commodity forward markets.
\newblock {\em Available at SSRN 3835053\/} (2021).

\bibitem{BS18}
{\sc {Benth}, F.~E., and {Simonsen}, I.~C.}
\newblock {The Heston stochastic volatility model in Hilbert space}.
\newblock {\em {Stochastic Anal. Appl.} 36}, 4 (2018), 733--750.

\bibitem{Bla76}
{\sc Black, F.}
\newblock The pricing of commodity contracts.
\newblock {\em {J. financ. econ.} 3}, 1-2 (1976), 167--179.

\bibitem{CM99}
{\sc Carr, P., and Madan, D.~B.}
\newblock {Option valuation using the fast Fourier transform}.
\newblock {\em J. Comput. Finance 2\/} (1999), 61--73.

\bibitem{CM87}
{\sc {Chojnowska-Michalik}, A.}
\newblock {On processes of Ornstein-Uhlenbeck type in Hilbert space}.
\newblock {\em {Stochastics} 21\/} (1987), 251--286.

\bibitem{cox2020affine}
{\sc Cox, S., Karbach, S., and Khedher, A.}
\newblock Affine pure-jump processes on positive {H}ilbert-{S}chmidt operators,
  2020.
\newblock Preprint, available at arXiv:2012.10406.

\bibitem{cox2021infinitedimensional}
{\sc Cox, S., Karbach, S., and Khedher, A.}
\newblock An infinite-dimensional affine stochastic volatility model.
\newblock {\em {Math. Finance}\/} (to appear 2022).
\newblock Preprint, available at arXiv:2108.02604.

\bibitem{Cro08}
{\sc {Crosby}, J.}
\newblock {Pricing a class of exotic commodity options in a multi-factor
  jump-diffusion model}.
\newblock {\em {Quant. Finance} 8}, 5 (2008), 471--483.

\bibitem{cuchiero2011affine}
{\sc Cuchiero, C.}
\newblock {\em Affine and polynomial processes}.
\newblock PhD thesis, ETH Zurich, 2011.

\bibitem{CFMT11}
{\sc {Cuchiero}, C., {Filipovi\'c}, D., {Mayerhofer}, E., and {Teichmann}, J.}
\newblock {Affine processes on positive semidefinite matrices}.
\newblock {\em {Ann. Appl. Probab.} 21}, 2 (2011), 397--463.

\bibitem{CT20}
{\sc {Cuchiero}, C., and {Teichmann}, J.}
\newblock {Generalized Feller processes and Markovian lifts of stochastic
  Volterra processes: the affine case}.
\newblock {\em {J. Evol. Equ.} 20}, 4 (2020), 1301--1348.

\bibitem{DL06}
{\sc {Dawson}, D.~A., and {Li}, Z.}
\newblock {Skew convolution semigroups and affine Markov processes}.
\newblock {\em {Ann. Probab.} 34}, 3 (2006), 1103--1142.

\bibitem{doersek2010semigroup}
{\sc Doersek, P., and Teichmann, J.}
\newblock A semigroup point of view on splitting schemes for stochastic
  (partial) differential equations, 2010.

\bibitem{DFS03}
{\sc {Duffie}, D., {Filipovi\'c}, D., and {Schachermayer}, W.}
\newblock {Affine processes and applications in finance}.
\newblock {\em {Ann. Appl. Probab.} 13}, 3 (2003), 984--1053.

\bibitem{EN00}
{\sc {Engel}, K.-J., and {Nagel}, R.}
\newblock {\em {One-parameter semigroups for linear evolution equations}},
  vol.~194.
\newblock Berlin: Springer, 2000.

\bibitem{Fil01}
{\sc {Filipovi\'c}, D.}
\newblock {\em {Consistency problems for Heath-Jarrow-Morton interest rate
  models}}, vol.~1760.
\newblock Berlin: Springer, 2001.

\bibitem{FM09}
{\sc {Filipovi\'c}, D., and {Mayerhofer}, E.}
\newblock {\em {Affine diffusion processes: theory and applications}}.
\newblock Berlin: de Gruyter, 2009.

\bibitem{FTT10}
{\sc {Filipovi\'c}, D., {Tappe}, S., and {Teichmann}, J.}
\newblock {Term structure models driven by Wiener processes and Poisson
  measures: existence and positivity}.
\newblock {\em {SIAM J. Financ. Math.} 1\/} (2010), 523--554.

\bibitem{Fri22}
{\sc Friesen, M.}
\newblock Long-time behavior for subcritical measure-valued branching processes
  with immigration.
\newblock {\em Pot. Anal.}, to appear (2022).

\bibitem{FJ20}
{\sc {Friesen}, M., and {Jin}, P.}
\newblock {On the anisotropic stable JCIR process}.
\newblock {\em {ALEA, Lat. Am. J. Probab. Math. Stat.} 17}, 2 (2020), 643--674.

\bibitem{FJKR20}
{\sc {Friesen}, M., {Jin}, P., {Kremer}, J., and {R\"udiger}, B.}
\newblock {Ergodicity of affine processes on the cone of symmetric positive
  semidefinite matrices}.
\newblock {\em {Adv. Appl. Probab.} 52}, 3 (2020), 825--854.

\bibitem{FJKR22}
{\sc Friesen, M., Jin, P., Kremer, J., and R\"{u}diger, B.}
\newblock Regularity of transition densities and ergodicity for affine
  jump-diffusion processes.
\newblock {\em Math. Nach.}, to appear (2022).

\bibitem{FJR20}
{\sc {Friesen}, M., {Jin}, P., and {R\"udiger}, B.}
\newblock {Stochastic equation and exponential ergodicity in Wasserstein
  distances for affine processes}.
\newblock {\em {Ann. Appl. Probab.} 30}, 5 (2020), 2165--2195.

\bibitem{GKK10}
{\sc {Glasserman}, P., and {Kim}, K.-K.}
\newblock {Moment explosions and stationary distributions in affine diffusion
  models}.
\newblock {\em {Math. Finance} 20}, 1 (2010), 1--33.

\bibitem{GS10}
{\sc Gourieroux, C., and Sufana, R.}
\newblock Derivative pricing with wishart multivariate stochastic volatility.
\newblock {\em Journal of Business \& Economic Statistics 28}, 3 (2010),
  438–451.

\bibitem{Gra16}
{\sc Grafendorfer, G.}
\newblock {\em Infinite-Dimensional Affine Processes}.
\newblock PhD thesis, ETH Zürich, 2016.

\bibitem{HKRS17}
{\sc {Hubalek}, F., {Keller-Ressel}, M., and {Sgarra}, C.}
\newblock {Geometric Asian option pricing in general affine stochastic
  volatility models with jumps}.
\newblock {\em {Quant. Finance} 17}, 6 (2017), 873--888.

\bibitem{JR13}
{\sc {Jacquier}, A., and {Roome}, P.}
\newblock {The small-maturity Heston forward smile}.
\newblock {\em {SIAM J. Financ. Math.} 4\/} (2013), 831--856.

\bibitem{JR15}
{\sc {Jacquier}, A., and {Roome}, P.}
\newblock {Asymptotics of forward implied volatility}.
\newblock {\em {SIAM J. Financ. Math.} 6\/} (2015), 307--351.

\bibitem{JKR17}
{\sc {Jin}, P., {Kremer}, J., and {R\"udiger}, B.}
\newblock {Exponential ergodicity of an affine two-factor model based on the
  \({\alpha}\)-root process}.
\newblock {\em {Adv. Appl. Probab.} 49}, 4 (2017), 1144--1169.

\bibitem{JKR19}
{\sc Jin, P., Kremer, J., and R\"{u}diger, B.}
\newblock Moments and ergodicity of the jump-diffusion {CIR} process.
\newblock {\em Stochastics 91}, 7 (2019), 974--997.

\bibitem{JKR20}
{\sc Jin, P., Kremer, J., and R{\"u}diger, B.}
\newblock {Existence of limiting distribution for affine processes}.
\newblock {\em {J. Math. Anal. Appl.} 486}, 2 (2020), 30.
\newblock Id/No 123912.

\bibitem{KMK10}
{\sc {Kallsen}, J., and {Muhle-Karbe}, J.}
\newblock {Exponentially affine martingales, affine measure changes and
  exponential moments of affine processes}.
\newblock {\em {Stochastic Processes Appl.} 120}, 2 (2010), 163--181.

\bibitem{KMKV11}
{\sc {Kallsen}, J., {Muhle-Karbe}, J., and {Vo{\ss}}, M.}
\newblock {Pricing options on variance in affine stochastic volatility models}.
\newblock {\em {Math. Finance} 21}, 4 (2011), 627--641.

\bibitem{keller2008affine}
{\sc Keller-Ressel, M.}
\newblock {\em Affine processes: theory and applications in finance}.
\newblock PhD thesis, TU Wien, 2008.

\bibitem{Kel11}
{\sc {Keller-Ressel}, M.}
\newblock {Moment explosions and long-term behavior of affine stochastic
  volatility models}.
\newblock {\em {Math. Finance} 21}, 1 (2011), 73--98.

\bibitem{KM15}
{\sc {Keller-Ressel}, M., and {Mayerhofer}, E.}
\newblock {Exponential moments of affine processes}.
\newblock {\em {Ann. Appl. Probab.} 25}, 2 (2015), 714--752.

\bibitem{KRM12}
{\sc {Keller-Ressel}, M., and {Mijatovi\'c}, A.}
\newblock {On the limit distributions of continuous-state branching processes
  with immigration}.
\newblock {\em {Stochastic Processes Appl.} 122}, 6 (2012), 2329--2345.

\bibitem{KN05}
{\sc {Kruse}, S., and {N\"ogel}, U.}
\newblock {On the pricing of forward starting options in Heston's model on
  stochastic volatility}.
\newblock {\em {Finance Stoch.} 9}, 2 (2005), 233--250.

\bibitem{LT08}
{\sc Leipold, M., and Trojani, F.}
\newblock {Asset Pricing with Matrix Jump Diffusions}, 2008.
\newblock Available at SSRN: https://ssrn.com/abstract=1274482.

\bibitem{LMP22}
{\sc Lemaire, V., Montes, T., and Pagès, G.}
\newblock Stationary heston model: calibration and pricing of exotics using
  product recursive quantization.
\newblock {\em Quantitative Finance 0}, 0 (2022), 1--19.

\bibitem{LV98}
{\sc {Lemmert}, R., and {Volkmann}, P.}
\newblock {On the positivity of semigroups of operators}.
\newblock {\em {Commentat. Math. Univ. Carol.} 39}, 3 (1998), 483--489.

\bibitem{Lew70}
{\sc {Lewis}, D.~R.}
\newblock {Integration with respect to vector measures}.
\newblock {\em {Pac. J. Math.} 33\/} (1970), 157--165.

\bibitem{Li21}
{\sc {Li}, Z.}
\newblock {Ergodicities and exponential ergodicities of Dawson-Watanabe type
  processes}.
\newblock {\em {Theory Probab. Appl.} 66}, 2 (2021), 276--298.

\bibitem{LM15}
{\sc {Li}, Z., and {Ma}, C.}
\newblock {Asymptotic properties of estimators in a stable Cox-Ingersoll-Ross
  model}.
\newblock {\em {Stochastic Processes Appl.} 125}, 8 (2015), 3196--3233.

\bibitem{MSV20}
{\sc {Mayerhofer}, E., {Stelzer}, R., and {Vestweber}, J.}
\newblock {Geometric ergodicity of affine processes on cones}.
\newblock {\em {Stochastic Processes Appl.} 130}, 7 (2020), 4141--4173.

\bibitem{Mer89}
{\sc {Merkle}, M.~J.}
\newblock {On weak convergence of measures on Hilbert spaces}.
\newblock {\em {J. Multivariate Anal.} 29}, 2 (1989), 252--259.

\bibitem{PS09}
{\sc {Pigorsch}, C., and {Stelzer}, R.}
\newblock {On the definition, stationary distribution and second order
  structure of positive semidefinite Ornstein-Uhlenbeck type processes}.
\newblock {\em {Bernoulli} 15}, 3 (2009), 754--773.

\bibitem{Ros56}
{\sc {Rosenblum}, M.}
\newblock {On the operator equation \(BX - XA = Q\)}.
\newblock {\em {Duke Math. J.} 23\/} (1956), 263--269.

\bibitem{STY20}
{\sc {Schmidt}, T., {Tappe}, S., and {Yu}, W.}
\newblock {Infinite dimensional affine processes}.
\newblock {\em {Stochastic Processes Appl.} 130}, 12 (2020), 7131--7169.

\bibitem{Vil09}
{\sc {Villani}, C.}
\newblock {\em {Optimal transport. Old and new}}, vol.~338.
\newblock Berlin: Springer, 2009.

\bibitem{Wan12}
{\sc Wang, J.}
\newblock On the exponential ergodicity of l{\'e}vy-driven ornstein-uhlenbeck
  processes.
\newblock {\em J. Appl. Probab. 49\/} (2012), 990--1004.

\bibitem{Wer00}
{\sc Werner, D.}
\newblock {\em Funktionalanalysis}, extended~ed.
\newblock Springer-Verlag, Berlin, 2000.

\end{thebibliography}
\end{document}